\documentclass[preprint,3p,authoryear,11pt,fleqn]{elsarticle}
\usepackage[round]{natbib}

\usepackage{amsmath,amsthm}
\usepackage{amssymb}
\usepackage{dsfont}
\usepackage{amsfonts}
\usepackage{enumerate}
\usepackage{mathrsfs}
\usepackage{multirow}
\usepackage{float}
\usepackage{graphicx}
\usepackage{color}
\usepackage{abbreviationsSSC-ext}
\journal{?}	

\begin{document}
\begin{frontmatter}

\title{Robust change-point detection in panel data}

\author[AD]{Alexander D{\"u}rre\corref{ad.cor}}
\address[AD]{Fakult\"at Statistik, Technische Universit\"at Dortmund, 44221 Dortmund, Germany}
\author[AD]{Roland Fried}

\cortext[ad.cor]{corresponding author: alexander.duerre@udo.edu, +49-231-755-4288}

\begin{abstract} 
In panel data we observe a usually high number $N$ of individuals over a time period $T$. Even if $T$ is large one often assumes stability of the model over time. We propose a nonparametric and robust test for a change in location and derive its asymptotic distribution under short range dependence and for $N,T\rightarrow \infty$. Some simulations show its usefulness under heavy tailed distributions.
\end{abstract}

\end{frontmatter}


\section{Introduction}
There is an increasing amount of literature on panel data taking structural breaks into account, see for example \cite{im2005}, \cite{Bai2009}, \cite{karavias2012} and \cite{baltagi2013}. Often the work of \cite{joseph1992} is regarded as starting point of change point detection in panel data. Therein two different change point models are introduced, namely the common change point model, where the time of change $\theta_i \in \{1,\ldots,T\}$ is identical for every individual $i=1,\ldots,N$, and the random change point model, where $\theta_i$ is independent and identically distributed following an unknown distribution $P_\theta.$ Recent work mostly considers the first approach. \cite{joseph1992}, \cite{wachter2005}, \cite{wachter2012} and \cite{baltagi2015b} consider homogeneous panels, where either the dependence structure or the noise distribution is the same for all individuals, whereas \cite{Bai} and \cite{kim2011} look also at heterogeneous ones.\\
Surprisingly little attention has been paid to robust change-point procedures. To the best of our knowledge, the only exception is the article of \cite{joseph1992} where a robust variation of their original test procedure, a kind of Mann-Whitney-Wilcoxon test, is proposed. The lack of robust methods is in contrast to change point detection in one-dimensional \citep[see for example][]{csorgo1987,huskova1996,dehling2013change} or multi-dimensional time series \citep[see][]{koziol1978,quessy2013,vogel2015robust}. In these settings robust procedures are not only more reliable in case of some corrupted observations, but they also turn out to be more powerful under heavy tailed distributions \citep[see][]{dehling2012testing,dehling2015robust}.\\
In this article we propose a robust test for the fundamental problem of a common change point in location. As opposed to \cite{joseph1992} our test can cope with heterogeneous panels and short range dependence. Like \cite{Bai}, \cite{Horvath} and \cite{jirak2015uniform} we consider the case where both the time dimension $T$ and the cross-sectional dimension $N$ tend to infinity. In contrast to them, by choosing a bounded $\psi$-function, moment assumptions are not required. Our test is based on M estimation and can be seen as a generalization of the test proposed in \cite{Horvath}. \\
This paper is structured as follows: in the next section we define the test statistic and explain how to choose the required tuning parameters. We present theoretical properties and conditions in Section 3. Section 4 contains a small simulation study showing the usefulness of the procedure and its finite sample performance. All proofs are deferred to an appendix.
\section{Testing procedure}
Let $(X_{i,t})_{i \in \{1,\ldots,N\},~t\in \{1,\ldots,T \}}$ denote a panel where $N$ is the number of individuals, which are observed at $T$ equidistant time points. We assume the simple structure 
\begin{align*}
X_{i,t}=\eta_i+\delta_iI_{t>t_0}+\epsilon_{i,t},~~i=1,\ldots,N,~t=1,\ldots,T,
\end{align*}
implying that the outcome only depends on an individual location $\eta_i$, an individual level-shift $\delta_i$ at a common time point $t_0$ and a random error $\epsilon_{i,t}$. In this setting we test the null hypothesis of a stationary panel
\begin{align*}
H_0:~\delta_1,\ldots,\delta_N=0
\end{align*}
against the alternative of a structural break at an unknown time point $t_0:$
\begin{align*}
H_1:~\exists i \in \{1,\ldots,N\} \mbox{~such that~}\delta_i\neq 0.
\end{align*}
The individual error processes $(\epsilon_{i,t})_{t=1,\ldots,T},~i=1,\ldots,N$ are supposed to be stationary and independent of each other. They are allowed to have different distributions and to be short range dependent. For technical assumptions on $(\epsilon_{i,t})_{i\in \{1,\ldots,N \},~t\in \{1,\ldots,T\}}$ see Section 3. To avoid identification problems we set $\mbox{median}(\epsilon_{i,1})=0,~i=1,\ldots,N.$ \\
In the following paragraph we develop the proposed test procedure. If one is only interested in detecting a change in one individual $i$ it is quite common to look at its CUSUM statistic
\begin{align}\label{cusump1}
Z^{(i)}_T(x)=\frac{1}{\sqrt{T}\nu_i}\left(\sum_{t=1}^{\lfloor Tx \rfloor}X_{i,t}-\frac{\lfloor Tx \rfloor}{T}\sum_{t=1}^T X_{i,t}\right),~~x\in[0,1],
\end{align}
which basically compares the mean value of the first part with that of the second for every split point. A large absolute difference for any split point indicates a structural change. If there is serial dependence the CUSUM statistic (\ref{cusump1}) depends on the so called long run variance
\begin{align*}\nu_i^2=Var(X_{i,1})+2\sum_{h=1}^\infty \mbox{Cov}(X_{i,1},X_{i,1+h}).\end{align*}
Since (\ref{cusump1}) is a linear statistic it is sensitive regarding unusually small or large values. To bound the influence of outliers one can transform the observations with a so called $\Psi-$function
\begin{align*}
Y_{i,t}=\Psi_i\left(\frac{X_{i,t}-\mu_i}{\sigma_i}\right),~~i=1,\ldots,N,~t=1,\ldots,T
\end{align*}
with $\Psi_i:\mathbb{R}\rightarrow \mathbb{R},~\mu_i\in \mathbb{R},~\sigma_i>0,~i=1,\ldots,N.$
Two examples and their impact on a short time series are shown in Figure \ref{psifun}. In change-point analysis monotone $\Psi-$functions are preferred \citep{huskova2012} since redescending ones not only limit the influence of unusual values, which can be a result of a large level shift, but also can shrink them to 0 as we can can see in Figure \ref{psifun} on the right.\\
\begin{figure}[H]
\includegraphics[width=0.95\textwidth]{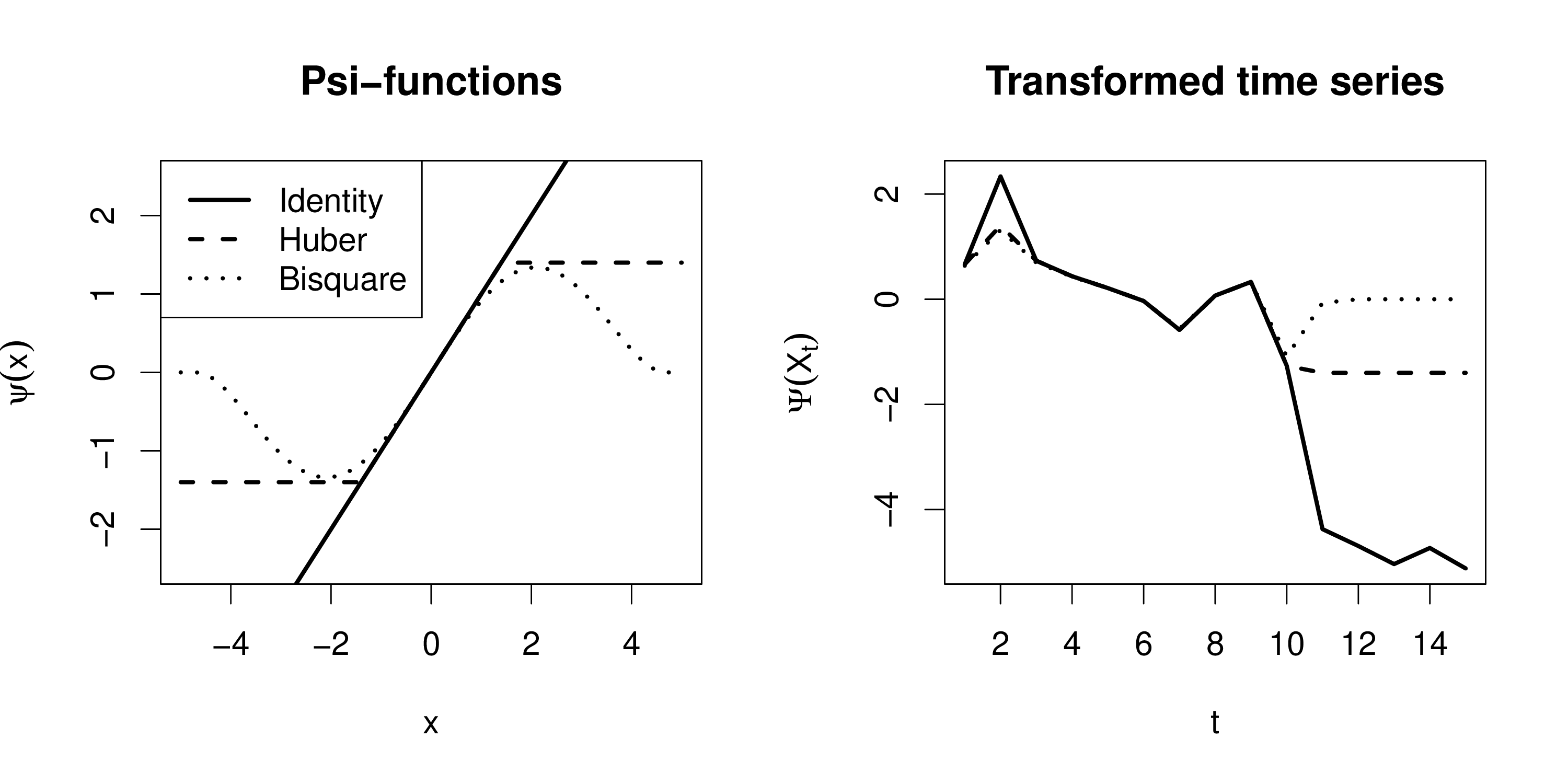}
\caption{$\Psi$-functions (left) and corresponding transformed time series (right) with a level shift at $t=11$.\label{psifun}}
\end{figure}
Usually the parameters $\mu_i,~i=1,\ldots,N$ describe the central location and $\sigma_i,~i=1,\ldots,N$ the scale. Their aim is to standardize the data such that outliers are treated independently of the scale and location of the underlying distribution. Regarding the change-point problem they can also be seen as tuning parameters, since the resulting test is valid under some restrictions we will state in Section 3 irrespective of their particular choice. But of course some choices are more suitable than others. If $\sigma_i$ is too large, outliers are hardly downweighted, and if it is too small, a lot of information is lost, see Figure \ref{sigmai}. An inappropriate value of $\mu_i$ can heavily skew formerly symmetric data and even destroy the data completely, if $\mu_i$ is far away from the observed data, see Figure \ref{sigmai} on the right. We will show in Section 3 that under some regularity conditions we can choose $\mu_i$ and $\sigma_i,~i=1,\ldots,N$ data adaptively. We recommend to use highly robust and computationally fast estimators for location respectively scale, like the median and the median absolute deviation (MAD).
\begin{figure}[H]
\includegraphics[width=0.95\textwidth]{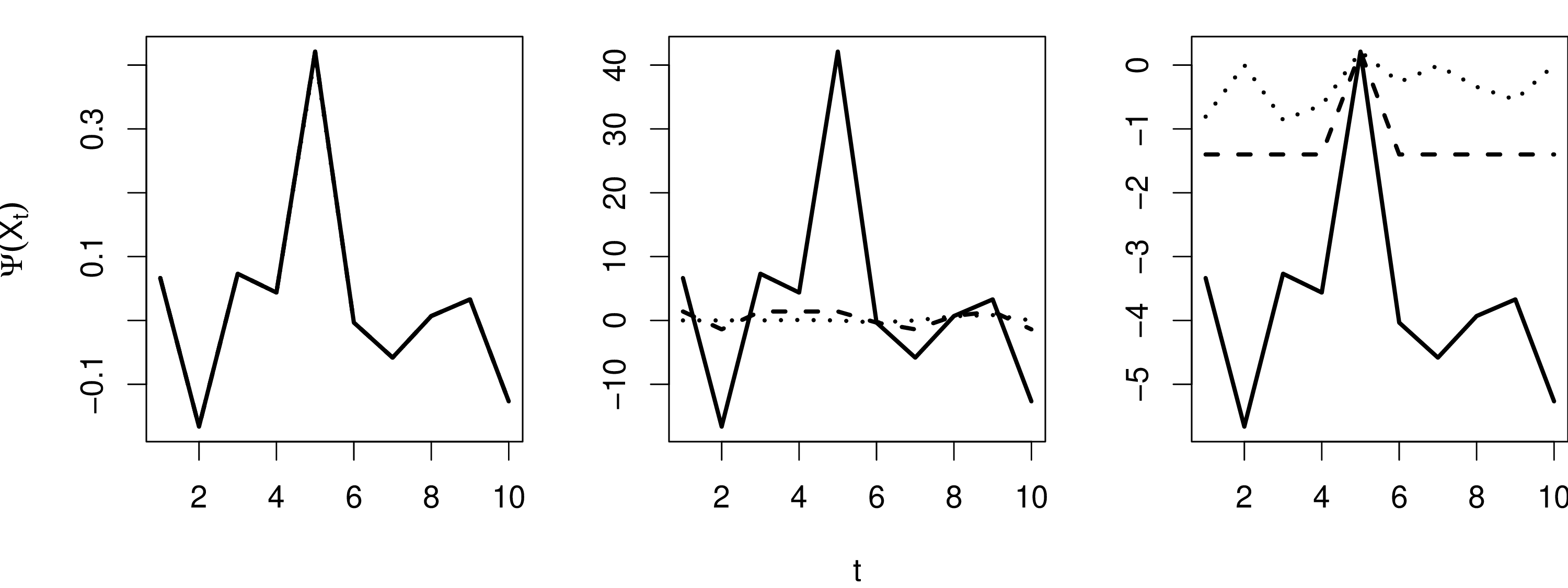}
\caption{Transformed time series using the identity (solid), Huber (dashed) and bisquare (dotted) $\Psi-$function with $\mu=0,~\sigma=10$ (left), $\mu=0,~\sigma=0.1$ (middle) and $\mu=4,~\sigma=1$ (right).\label{sigmai}}
\end{figure}
The CUSUM statistic of the transformed panel is then defined as
\begin{align*}
S^{(i)}_T(x)=\frac{1}{\sqrt{T}v_i}\left(\sum_{t=1}^{\lfloor Tx \rfloor}Y_{i,t}-\frac{\lfloor Tx \rfloor}{T}\sum_{t=1}^T Y_{i,t}\right),~~x\in[0,1]
\end{align*}
with
\begin{align*}v_i^2=Var(Y_{i,1})+2\sum_{h=1}^\infty \mbox{Cov}(Y_{i,1},Y_{i,1+h}).\end{align*} 
In the one dimensional case $N=1$, setting $\sigma_1=1$ and choosing $\mu_1$ as the M-estimator corresponding to the chosen $\Psi-$function this statistic was proposed by \cite{huskova2012} for change point detection.\\
Assuming independence between the individuals the related multivariate Wald-type test for finite $N$ equals
\begin{align}\label{Wald}
\sum_{i=1}^N \left(S^{(i)}_T(x)\right)^2,~~x\in[0,1],
\end{align}
which converges for $T\rightarrow \infty$ under some regularity conditions to the sum of independent squared Brownian bridges. If additionally $N\rightarrow \infty$ one has to normalize (\ref{Wald}):
\begin{align*}
W_{N,T}(x)=\frac{1}{\sqrt{N}}\sum_{i=1}^N \left(\left(S_T^{(i)}(x)\right)^2-\frac{\lfloor x T \rfloor(T-\lfloor x T \rfloor)}{T^2}\right),~~x\in [0,1].
\end{align*}
Note that $\frac{\lfloor x T \rfloor(T-\lfloor x T \rfloor)}{T^2}$ nearly equals the variance of a Brownian bridge and therefore approximates the mean of $S_T^{(i)}(x)$ for large $T$. Setting $\Psi_i(x)=x,~i=1,\ldots,N,~\mu_i$ cancels out, $\sigma_i$ is absorbed into $v_i$ and one arrives at the non-robust panel-cusum-statistic which was originally proposed by \cite{Bai} and investigated theoretically by \cite{Horvath}. If $T$ tends faster to infinity than $N$ (the accurate rates can be found in Section 3) and under some regularity conditions $W_{N,T}(x)$ converges weakly to a Gaussian process $\Gamma(x)$ defined by
\begin{align*}
\mathbb{E}(\Gamma(x))=0~~\mbox{and}~~\mbox{Cov}(\Gamma(x),\Gamma(y))=x^2(1-y^2),~0\leq x\leq y\leq 1.
\end{align*} 
It is shown in \cite{Horvath} that such a process can be simulated based on a standard Brownian motion $(B(t))_{0\leq t<\infty}$ using the following relationship:
\begin{align*}
\{\Gamma(x),~0\leq x\leq 1\}\stackrel{D}{=}\left\{\sqrt{2}(1-x)^2 B \left(\frac{x^2}{1-x^2} \right),~0\leq x\leq 1 \right\}.
\end{align*}
One rejects the null hypothesis of a stationary panel if
\begin{align}\label{tstat}
\sup_{0<x<1} |W_{N,T}(x)|
\end{align}
exceeds a certain quantile of $\sup_{0<x<1}|\Gamma(x)|.$ A small selection of critical values can be found in Table \ref{tablequan}.
\begin{table}[H]
\begin{center}
\begin{tabular}{l|ccccc}
$\alpha$&0.9&0.95&0.975&0.99&0.995\\\hline
$q_\alpha$&0.899&0.990&1.072&1.173&1.245
\end{tabular}
\caption{Quantiles of $\sup_{0<x<1}|\Gamma(x)|.$\label{tablequan}}
\end{center}
\end{table}
The proposed test statistic (\ref{tstat}) has the largest power if the change occurs in the middle of the sample. If one needs a high power near the margins, one can look at other functionals of $W_{N,T}(x)$, for example $\int_0^1 |W_{N,T}(x)|~dx.$\\
In practice $v_i,~i=1,\ldots,N,$ is unknown and has to be estimated. We propose to use a kernel estimator. Denote by
\begin{align*}
\hat{\gamma}_i(h)=\frac{1}{T}\sum_{t=1}^{N-h}\left(Y_{i,t}-\overline{Y}_i\right)\left(Y_{i,t+h}-\overline{Y}_i\right)
\end{align*}
with $\overline{Y}_i=\frac{1}{T}\sum_{t=1}^T Y_{i,t}$ the empirical autocovariance of $(Y_{i,t})_{t=1,\ldots,T}$ of lag $h,~h=1,\ldots,T-1$. Let furthermore $k:\mathbb{R}\rightarrow [-1,1]$ be a kernel function and $b_T$ a bandwidth, then
\begin{align}\label{kernel}
\hat{v}_i=\hat{\gamma}_{i,0}+2\sum_{h=0}^{b_{T}}\hat{\gamma}_{i}(h)k\left(\frac{h}{b_{T}}\right)
\end{align}
represents the related kernel estimator for $v_i.$ Theoretical conditions on $k$ and $b_T$ can be found in the next section. Simulations in Section 4 indicate that the flat top kernel $k=k_F$
\begin{align}\label{flattop}
k_F(x)=\begin{cases}
1&|x|\leq 0.5\\
2-2|x|&0.5<|x|\leq 1\\
0&|x|>1
\end{cases}
\end{align}
with $b_T=T^{0.4}$ works well if the serial dependence is not very large.
\section{Theoretical results}
In this section we give theoretical justification of the testing procedure and compile all conditions which are necessary for the asymptotical results. We start by defining the type of short range dependence we impose on the error processes $(\epsilon_{i,t})_{t=1,\ldots,T},~i=1,\ldots,N$.
We assume that they are near epoch dependent in probability (P-NED) on an absolutely regular process.
\begin{definition}
\begin{enumerate}[i)]
\item Let $A,B\subset{F}$ be two $\sigma-$fields on the probability space $(\Omega,F,P)$. Then the regularity coefficient of $A$ and $B$ is defined as
\begin{align*}
\beta(A,B)=\mathbb{E}\left(\sup_{M\in A}|P(M|B)-P(M)|\right).
\end{align*}
\item For a stationary process $(Z_t)_{t\in \mathbb{Z}}$, the absolute regularity coefficients are given by
\begin{align*}
\beta_k=\sup_{n\in \mathbb{N}}\beta(F_1^n,F_{n+k}^\infty),
\end{align*}
where $F_n^k=\sigma(Z_n,\ldots,Z_k)$ denotes the $\sigma-$field generated by $Z_n,\ldots,Z_k$. A process is called absolutely regular, if $\beta_k\rightarrow 0.$\\
\item The process $(X_t)_{t \in \mathbb{N}}$ is P-NED on the absolutely regular process $(Z_t)_{t\in \mathbb{Z}}$, if there is a sequence of approximating constants $(a_k)_{k\in \mathbb{N}}$ fulfilling $a_k\rightarrow 0$ for $k\rightarrow \infty$, a sequence of functions $f_k:\mathbb{R}^{ k}\rightarrow \mathbb{R}$ for $k\in \mathbb{N}$ and a decreasing function $\Phi:~(0,\infty)\rightarrow (0,\infty)$ such that
\begin{align}\label{PNEDprop}
P(|X_0-f_k(Z_{-k},\ldots,Z_k)|>\epsilon)\leq a_k\Phi(\epsilon)
\end{align}
\end{enumerate}
\end{definition}
\begin{remark}
\begin{itemize}
\item Looking at functionals of processes is quite natural, in fact ARMA and GARCH models are defined as functionals of iid sequences. Part \textit{iii)} demands that the influence of observations of the underlying process $(Z_t)_{t\in \mathbb{Z}}$, corresponding to time points far away from the observed $X_t$, is small and even vanishes, if the time difference between them tends to infinity.
\item The underlying process $(Z_t)_{t\in \mathbb{Z}}$ is sometimes assumed to be iid, see for example \cite{hormann2008} and \cite{aue2009break}, and though this class would yet be quite general, this assumption is in our case more restrictive than necessary. A main component of our proofs are moment inequalities, which work fine for absolutely regular processes \citep{philipp1986invariance}.
\item The concept of near epoch dependence has a long history. It was introduced in \cite{ibragimov1962} where the difference between the approximating functional $f_k(Z_{-k},\ldots,Z_k)$ and $X_0$ was measured by the $L_2$ norm, which requires existing second moments of $(X_t)_{t \in \mathbb{N}}$. With applications under heavy tailed distributions and robustness in mind this is somehow restrictive, which was the reason \cite{Bierens} proposed substituting the so called $L_0$ for the $L_2$ norm. We use here the variation presented by \cite{dehling2012testing} which adds the error function $\Phi.$
\item Let $\Psi$ be Lipschitz continuous with Lipschitz constant $L$, then it is easy to see that with $(X_t)_{t \in \mathbb{N}}$ also $(\Psi(X_t))_{t\in \mathbb{N}}$ is P-NED on the same underlying process $(Z_t)_{t\in \mathbb{Z}}$ with approximating functionals $\tilde{f}_k=\Psi\circ f_k$ and constants $(a_k)_{k\in \mathbb{Z}}$. For the error function corresponding to $(\Psi(X_t))_{t\in \mathbb{N}}$ we have $\tilde{\Phi}(x)=\Phi(x/L)$. Furthermore, if $\Psi$ is bounded by a constant $c$: $\tilde{\Phi}(|2c+\epsilon|)=0,~\epsilon>0.$
\end{itemize}
\end{remark}
Before stating our first result, we compile some assumptions on the errors $(\epsilon_{i,t})_{i=1,\ldots,N,~t=1,\ldots,T}$, which we need repeatedly.
\begin{assumption}\label{Grundassum}
Let $(\epsilon_{i,t})_{t=1,\ldots,T}$ be stationary processes, which are 
P-NED on absolutely regular processes $(Z_{i,t})_{t\in \mathbb{Z}}$ with approximating constants $(a_{k,i})_{k\in \mathbb{N}},$ regularity coefficients $(\beta_{k,i})_{k\in \mathbb{N}}$ and error functions $\Phi_i$ for $i=1,\ldots,N$, which fulfil
\begin{enumerate}[1)]
\item $(\epsilon_{
1,t})_{t=1,\ldots,T},\ldots,(\epsilon_{N,t})_{t=1,\ldots,T}$ are independent processes,
\item \textbf{either} a): 
\begin{enumerate}[i)]
\item there exists $c_1$ such that  $\sup_{i\in \mathbb{N}}|Y_{i,1}|<c_1~\mbox{a.s.},$
\item there exist $c_2>0$ and $b>8$ such that $\sup_{i\in \mathbb{N}}a_{k,i},\beta_{k,i}\leq c_2(1+k)^{-b},$
\item there exists $\Phi:\mathbb{R}_+\rightarrow \mathbb{R}_+$ with $\sup_{i\in \mathbb{N}}\Phi_i(x)\leq \Phi(x)~\forall x> 0$ such that $\int_0^1 \Phi(x)~dx<\infty,$
\end{enumerate}
\textbf{or}  b):
\begin{enumerate} [i)]
\item there exist $c_1>0$ and $a\geq 8$ such that $\sup_{i\in \mathbb{N}}\mathbb{E}\left(|Y_{i,1}|^a\right)<c_1,$
\item there exist $c_2>0,~b>8$ such that for $sup_{i \in \mathbb{N}}a_{k,i}\leq c_2(1+k)^{-\frac{b\cdot a}{a-7}} \mbox{~and~}
\sup_{i \in \mathbb{N}}\beta_{k,i}\leq c_2(1+k)^{-\frac{b\cdot a}{a-8}},$
\item there exists $\Phi:\mathbb{R}_+\rightarrow \mathbb{R}_+$ with $\sup_{i\in \mathbb{N}}\Phi_i(x)\leq \Phi(x)~\forall x> 0$ such that $\int_1^\infty x^{a-1}\Phi(x)~dx<\infty$ and $\int_0^1 x^{\frac{a}{a-7}-1}\Phi(x)~dx<\infty.$
\end{enumerate} 
\end{enumerate}
\end{assumption}
\begin{remark}
\begin{itemize}
\item Assumption $1)$ is very strong but essential in the proof. Deriving asymptotics under dependence between panels is for this test statistic if at all only possible under very restrictive assumptions. For example \cite{Horvath} add dependence through common time effects, which need to shrink with increasing $N$. Another possibility could be to impose a spatial dependence structure like in \cite{jirak2015uniform}.
\item Assumption $2)$ differentiates depending on whether the transformed random variables $Y_{i,t}$ are bounded or not. The former is the standard case for a robust procedure, since it requires bounded $\Psi-$ functions. The later contains the non-robust test with $\Psi_i(x)=x, i=1,\ldots,N$.
\item If one uses finitely many bounded $\Psi$-functions assumption 2) a) i) is always fulfilled regardless of the underlying distribution. This is a major advantage of the robust method over the non-robust one which depends on finite moments of order 8, see also \cite{Horvath}.
\item If $|\Psi_{i}(\epsilon_{i,1})|$ is not bounded one needs at least finite eighth moments of the transformed time series. If higher moment exists, one can weaken the dependence condition 2) b) ii). Assumption 2) b) iii) is quite technical and implies some finite moments of the error of the P-NED approximation. Using the P-NED concept is quite artificial in this case. Nevertheless we will prove the results also under these assumptions since it generalizes the results of \cite{Horvath} to non-linear processes.
\item Assumption 2 a)/b) ii) define that the dependence needs to decay uniformly. A similar condition can also be found in \cite{Horvath}.
\end{itemize}
\end{remark}


In the following theorem we derive the asymptotic behaviour of $W_{N,T}(x)$ if $v_i$ is known and $\sigma_i,\mu_i$ are fixed for $i=1,\ldots,N$. Let for this purpose $\stackrel{D[0,1]}{\rightarrow}$ denote weak convergence in the Skorokhod space $D[0,1].$
\begin{theorem}\label{T1}
Let $(\epsilon_{i,t})_{i=1,\ldots,N,~t=1,\ldots,T}$ fulfil Assumption 1 and furthermore
\begin{enumerate}[i)]
\item $N,T\rightarrow \infty$ with $N/T\rightarrow 0$
\item the $\Psi-$functions are Lipschitz continuous with constants $L_i$ which are bounded from above by $L$,
\item there exists $\delta$ such that $\inf_{i\in \mathbb{N}}v_i\geq \delta,$
\item there exists $\sigma_0$ such that $\inf_{i\in \mathbb{N}}\sigma_i\geq \sigma_0,$
\end{enumerate}
then
\begin{align*}
(W_N,T(x))_{x\in [0,1]}\stackrel{D[0,1]}{\rightarrow} (\Gamma(x))_{x\in [0,1]}. 
\end{align*}
\end{theorem}
\begin{remark}
\begin{itemize}
\item The special case $\Psi_i(x)=x$ is treated by \cite{Horvath}. Note that the functions $(\Psi_i)_{i \in \mathbb{N}}$ do not interfere with the limiting distribution as long as they do not cut existing moments or lead to imploding variances.
\item Condition $i)$ permits $N$ to grow somewhat faster than $T.$ Intuitively one might think that a large $N$ always improves the asymptotics, but one needs to remember that the standardization by $\frac{\lfloor x T \rfloor(T-\lfloor x T \rfloor)}{T^2}$ and $v_i$ is only an approximation for large $T$. If $N$ grows too fast these small errors sum up to something which is not negligible. In \cite{Horvath} only $N/T^2\rightarrow 0$ is required and it is not clear why we need this somewhat stronger assumption here, which is by the way only necessary for to verify tightness. A possible reason is our more flexible dependence structure of the errors. However, since the following Theorems also require $N/T\rightarrow 0,$ we do not judge this more restrictive condition as a real drawback.
\item Assumption $ii)$ implies some regularity conditions for the $\Psi$-functions. Similar conditions can be found in one dimensional change point detection, see \cite{huskova2012}. For usual $\psi$ functions this is no restriction.
\item Condition $iii)$ states that the variability within one individual should not tend to 0 and this condition can also be found in \cite{Horvath}. Though one has to mention that the condition is here formulated based on the transformed time series. Inappropriate choices of $\mu_i$ and $\sigma_i$ can shrink all values to a constant, see Figure \ref{sigmai}, and so cause the assumption to be violated, though the original time series was permissible.
\end{itemize}
\end{remark}
To actually apply the testing procedure, one has to find estimators for $v_i,~i=1,\ldots,N$. As usual the asymptotics depend on the smoothness of the kernel $k$ in 0. We need the following assumptions on the kernel.
\begin{assumption}\label{assukernel}
Let $k:\mathbb{R}\rightarrow [-1,1]$ be a kernel function with
\begin{enumerate}[i)]
\item $k(0)=1,$
\item $k(-x)=x, x\geq 0,$
\item $k(x)=0,~|x|\geq 1$
\item $k(x)\leq R,~\forall x\in \mathbb{R},$
\item the first $m-1>0$ derivatives of $k$ in 0 are 0,
\item $\exists \epsilon,M>0$ such that $|k^{(m)}(x)|\leq M, ~|x|\leq \epsilon$, where $k^{(m)}$ denotes the $m-$th derivative of $k.$
\end{enumerate} 
\end{assumption}
\begin{remark}
\begin{itemize}
\item While conditions i) and ii) are very conventional, see for example section 9.3.2. in \cite{anderson}, requirements iii) and iv) are more particular, though fulfilled for most kernels.
\item The number $m$ characterizes the smoothness of the kernel in 0 which determines the bias of the kernel estimator, see also chapter 9.3.2 of \cite{anderson}. The popular Daniell, Blackman-Tukey, Hanning, Hamming and Parzen kernels fulfil condition v) with $m=2$, the flat top kernel (\ref{flattop}) even for arbitrary $m\in \mathbb{N}$. Since the Bartlett-Kernel is not differentiable in 0, it is not covered by the assumptions.
\item Condition vi) is fulfilled for all above mentioned kernels except the Bartlett kernel. Both assumption v) and vi) are additional requirements, not necessary in the usual one- or multidimensional time series context. We need them here to ensure uniform convergence over all individuals.
\end{itemize}
\end{remark}
The next theorem deals with the case where the theoretical long run variances $v_i$ are replaced by their kernel estimations (\ref{kernel}).  
\begin{theorem}\label{T2}
Let $Y_{i,t}=\Psi_i\left(\frac{X_{i,t}-\mu_i}{\sigma_i}\right)$ and denote
\begin{align*}
\tilde{S}^{(i)}_T(x)=\frac{1}{\sqrt{T}\hat{v}_i}\left(\sum_{t=1}^{\lfloor Tx \rfloor}Y_{i,t}-\frac{\lfloor Tx \rfloor}{T}\sum_{t=1}^T Y_{i,t}\right),~~x\in[0,1],
\end{align*}
the CUSUM-statistic with estimated long run variance using (\ref{kernel}) and
\begin{align*}
\tilde{W}_{N,T}(x)=\frac{1}{\sqrt{N}}\sum_{i=1}^N \left(\left(\tilde{S}_T^{(i)}(x)\right)^2-\frac{\lfloor x T \rfloor(T-\lfloor x T \rfloor)}{T^2}\right),~~x\in [0,1].
\end{align*}
the related panel-CUSUM-statistic. Let additionally to the Assumptions of Theorem \ref{T1} \begin{align}\label{kernrate}Nb_T/T\rightarrow 0 \mbox{~and~} N/b_T^{2s}\rightarrow 0
~~\mbox{where}~~
s=\begin{cases}
\min(m,b-1)&m\neq b-1\\
m-1&m=b-1\end{cases},\end{align}
then
\begin{align*}
\sup_{x \in [0,1]}|\tilde{W}_{N,T}(x)-W_{N,T}(x)|\rightarrow 0.
\end{align*}
\end{theorem}
\begin{remark}
\begin{itemize}
\item The assumption $Nb_T/T\rightarrow 0$ seems to be stronger than necessary. In the same situation \cite{Horvath} only require $Nb_T^2/T^2\rightarrow 0$. The reason seems to be a difference in the proofs. While at one point \cite{Horvath} consider the $L_2$ norm, we choose the $L_1$ norm to make the proof somewhat more feasible. Remember that we allow for a little more complicated dependence structure.
\item As usual for kernel estimators the rate of $b_T$ should neither be too fast (large variance) nor too slow (large bias). This is reflected in (\ref{kernrate}). Furthermore, we see that from the theoretical point of view the flat-top kernel is preferable. In this case the bias condition $N/b_T^{2s}\rightarrow 0$ only depends on the strength of the serial dependence.
\end{itemize}
\end{remark}
Now we want to prove that the tuning parameters $\mu_i$ and $\sigma_i,~i=1\ldots,N$ can be chosen data-adaptively. Let therefore $\mu(\cdot)$ denote a univariate location measure, which implies $\mu(F^\star)=a+\mu(F)$ for $a,b\in \mathbb{R}$ and any one-dimensional distribution $F$ where $F^\star$ is the distribution of $a+bX$ with $X\sim F.$ Examples are the mean $\mu_{mean}(F)=\mathbb{E}(X)$ and the more robust median $\mu_{med}(F)=\mbox{median}(X).$ For a univariate scale measure $\sigma(\cdot)$ we demand $\sigma(F^\star)=|a|\sigma(F)$. Maybe the most popular representative is the standard deviation $\sigma_{SD}(F)=\sqrt{\mathbb{E}(X-\mathbb{E}(X))^2},$ which is of course not robust. A more appropriate choice here is the median absolute deviation $\sigma_{MAD}(F)=c_F\cdot \mbox{median}(|X-\mbox{median}(X)|)$, where often $c_F=1.4826$ so that $\sigma_{SD}(F)=\sigma_{MAD}(F)$ if $F$ is a normal distribution.\\
The related estimators $\hat{\mu}$ are usually the measure $\mu(\cdot)$ applied to the empirical distribution of the sample $(X_1,\ldots,X_T)$. Though there are sometimes differences. The standard deviation is for example often defined as $\sqrt{\frac{1}{T-1}\sum_{i=1}^T (X_i-\frac{1}{T}\sum_{i=1}^T X_i)^2}$ instead of $\sqrt{\frac{1}{T}\sum_{i=1}^T (X_i-\frac{1}{T}\sum_{i=1}^T X_i)^2}.$\\
From now on we define the standardization parameters $\mu_i$ and $\sigma_i$ as
\begin{align*}
\mu_i=\mu(F_i)~~\mbox{and}~~\sigma_i=\sigma(F_i),~i=1,\ldots,N,
\end{align*}
where $F_i$ is the marginal distribution of $X_{1,t}$ for some location measure $\mu(\cdot)$ and scale measure $\sigma(\cdot)$. The next Theorem gives conditions under which it is asymptotically negligible whether one knows these theoretical values or estimates them, as long as the estimators converge fast enough.

\begin{theorem}\label{storung}
Denote by
\begin{align*}
\check{S}^{(i)}_T(x)=\frac{1}{\sqrt{T}\check{v}_i}\left(\sum_{t=1}^{\lfloor Tx \rfloor}\Psi_i\left(\frac{X_{i,t}-\hat{\mu}_i}{\hat{\sigma}_i}\right)-\frac{\lfloor Tx \rfloor}{T}\sum_{t=1}^T \Psi_i\left(\frac{X_{i,t}-\hat{\mu}_i}{\hat{\sigma}_i}\right)\right),~~x\in[0,1]
\end{align*}
the CUSUM-statistic with estimated location and scale parameter, $\check{v}_i$ the long run variance estimation based on $\hat{\mu}_i$ and $\hat{\sigma}_i$ and
\begin{align*}
\check{W}_{N,T}(x)=\frac{1}{\sqrt{N}}\sum_{i=1}^N \left(\left(\check{S}_T^{(i)}(x)\right)^2-\frac{\lfloor x T \rfloor(T-\lfloor x T \rfloor)}{T^2}\right),~~x\in [0,1].
\end{align*}
the related panel-CUSUM-statistic.
Let additionally to the Assumptions of Theorem 4 hold that
\begin{enumerate}[i)]
\item $\Psi_i$ is $m+1$ times continuously differentiable with derivatives 
$\Psi_i^{(1)},\ldots,\Psi_i^{(m+1)}$ and there exists $c$ such that 
$\sup_{i\in \mathbb{N}}|\Psi^{(k)}_i(x)x^k|_{\infty}<c,~\mbox{for}~k=1,\ldots,m,$
\item there exist $d,\epsilon,\delta>0$ such that $\sup_{i\in \mathbb{N}}|\Psi_i^{(m+1)}(x(1+y)+z)x^k|_\infty\leq d~\forall |y|\leq \epsilon$, $|z|\leq \delta$, $k\leq m+1,$
\item there exists $c_1,c_2,\alpha,\beta>0$ such that for all $T\in \mathbb{N}$:
\begin{align*}
\sup_{i\in \mathbb{N}}P(|\hat{\mu}_i-\mu_i|>c_1T^{-\beta})\leq c_2T^{-\alpha}\mbox{~~and~~}\sup_{i\in \mathbb{N}}P(|\hat{\sigma}_i-\sigma_i|>c_1T^{-\beta})\leq c_2T^{-\alpha},
\end{align*}
\item $N/T^{2\beta}\rightarrow 0,$ $N/T^\alpha \rightarrow 0$ and $N/T^{2\beta(m+1)-1)}\rightarrow 0,$
\end{enumerate}
then \begin{align*}
\sup_{x \in [0,1]}|\check{W}_{N,T}(x)-W_{N,T}(x)|\rightarrow 0.
\end{align*}
\end{theorem}

\begin{remark}
\begin{itemize}
\item Assumption i) is often not fulfilled for standard $\Psi-$functions. The Huber $\Psi-$function is for example only continuous but in two points not differentiable. This assumption may be relaxed to the case where the function is differentiable in all but a finite number of points at least as the error distribution is continuous. This would be fulfilled for all common $\Psi-$functions. Nevertheless the current condition is not a restriction in practice, since one can always find $m+1$-times differentiable modifications which hardly differ from the original $\Psi-$function.\\The existence of the upper bounds $c$ and $d$ are also no substantial restriction, since $\Psi-$functions are usually constant for large values (and so its derivatives are 0).
\item Assumption iii) is the main restriction of the theorem. We will show in the next Theorem that one can find such tail probability bounds for median and MAD. Note that the parameter $\alpha$ and $\beta$ determine the rate of convergence as can be seen in Assumption iv).
\item Additionally to the persisting conditions on the ratio of $N$ and $T$, we get the assumptions $N/T^{2\beta}\rightarrow 0,$ $N/T^\alpha \rightarrow 0$ and $N/T^{\beta(m+1)-1}\rightarrow 0.$ If exponential inequalities for $\hat{\mu}_i$ and $\hat{\sigma}_i,~i=1,\ldots,N$ are available and $\Psi_i$ is 2 times continuously differentiable, $i=1,\ldots,N$ this boils down to $N/T^{1-\epsilon}\rightarrow 0$ for some $\epsilon> 0.$ So $T$ has to grow a little faster than $N.$
\end{itemize}
\end{remark}

Finally we want to investigate if our proposed estimators median and MAD fulfil assumption \textit{iii)} of Theorem \ref{storung} and how the parameters $\alpha$ and $\beta$ depend on the properties of the observed processes. We formulate the result for a one dimensional time series $(X_t)_{t\in \mathbb{N}}$ respectively its transformation $(Y_t)_{t\in \mathbb{N}}$ with $Y_t=|X_t-\mu_{med}|,~t\in \mathbb{N}.$
\begin{theorem}
Let $(X_t)_{t \in \mathbb{N}}$ be stationary and P-NED on an absolutely regular process $(Z_{t})_{t\in \mathbb{Z}}$ with approximation constants $(a_k)_{k\in \mathbb{N}}$, functionals $(f_k)_{k\in \mathbb{N}}$, error function $\Phi$ and absolutely regularity coefficients $(\beta_k)_{k\in \mathbb{N}}.$ Furthermore 
\begin{enumerate}[i)]
\item there exist $\kappa>0$ and $p\in \mathbb{N}$ even such that
$\sum_{k=1}^\infty (a_k\Phi(a_k^\kappa)+a_k^\kappa)k^p <\infty ~~\mbox{and}~~\sum_{k=1}^\infty \beta_k k^p < \infty,$
\item $X_1$ is continuous with bounded density $f$ and there exist $M,\epsilon>0$ with  $f(x)\geq M ~\mbox{for}~x\in(\mu_{med}-\epsilon,\mu_{med}+\epsilon),$
\end{enumerate}
then there exists $c$ such that\begin{align}\label{medianaus}
P(|\hat{\mu}_{med}-\mu_{med}|>T^{-\beta})\leq c T^{-p/2+p\beta}~\forall T\in \mathbb{N}.\end{align}
If additionally
\begin{enumerate}[iii)]
\item $Y_1$ is continuous with bounded density $g$ and there exist $M,\epsilon>0$ with  $g(x)\geq M ~\mbox{for}~x\in(\sigma_{MAD}/c_F-\epsilon,\sigma_{MAD}/c_F+\epsilon),$
\end{enumerate}
then there exists $c$ such that\begin{align*}
P(|\hat{\sigma}_{MAD}-\sigma_{MAD}|>T^{-\beta})\leq c T^{-p/2+p\beta}~\forall T\in \mathbb{N}.
\end{align*}
\end{theorem}
 
\begin{remark}
\begin{itemize}
\item Under independence there exist exponential inequalities for the median and the MAD, see \cite{serfling2009exponential}. We are not aware such exponential inequalities under dependence, but looking at the proof in \cite{serfling} or \cite{serfling2009exponential} it only depends on the existence of a Hoeffding inequality, which is proven under various dependence conditions, see \cite{kontorovich2008concentration} or \cite{kallabis2006exponential}. However, proving such inequalities under the P-NED condition is not the objective of this paper.
\item The condition of a continuous distribution might be relaxed if one slightly redefines median and MAD, see also \cite{serfling2009exponential}.
\end{itemize}
\end{remark}
\section{Simulation}
In this section we want to evaluate our proposed test statistic concerning two aspects: The size under the null hypothesis and the power under the alternative. There are quite different simulation scenarios possible, since we allow for quite diverse serial dependence structures and distributions of the individuals. For thee reason of comparison we orientate ourselves at the AR(1) models considered in \cite{Horvath}
\begin{align*}
X_{i,t}=\rho\cdot X_{i,t-1}+a_{i,t},~t=1,\ldots,T,~i=1,\ldots,N.
\end{align*}
We compare the non robust panel CUSUM statistic proposed by \cite{Horvath} with our robust alternative. We use a two times continuously differentiable version of Hubers-$\Psi-$function which is shown in Figure \ref{hubersmooth}. Furthermore we choose the flat top kernel (\ref{flattop}) with a bandwidth $b_T=T^{0.4}$ for both estimators. Note that \cite{Horvath} originally use a rather small value of $b_T$ varying between 2.5 and 5 which is, however, not appropriate in our simulations where we also look at relatively large $T$. Simulation results are based on 1000 runs each.\\
We first look at finite sample properties under the null hypothesis. In the case of $\rho=0$, where we have no serial dependence, both tests need at least $T=200$ to work properly as one can see in Table \ref{H01}. Furthermore the simulations confirm the theoretical results that $N$ must not grow much faster than $T$. We also notice that there is basically no difference between the tests under Gaussianity.\\
\begin{figure}[H]
\begin{center} 
\includegraphics[width=0.35\textwidth]{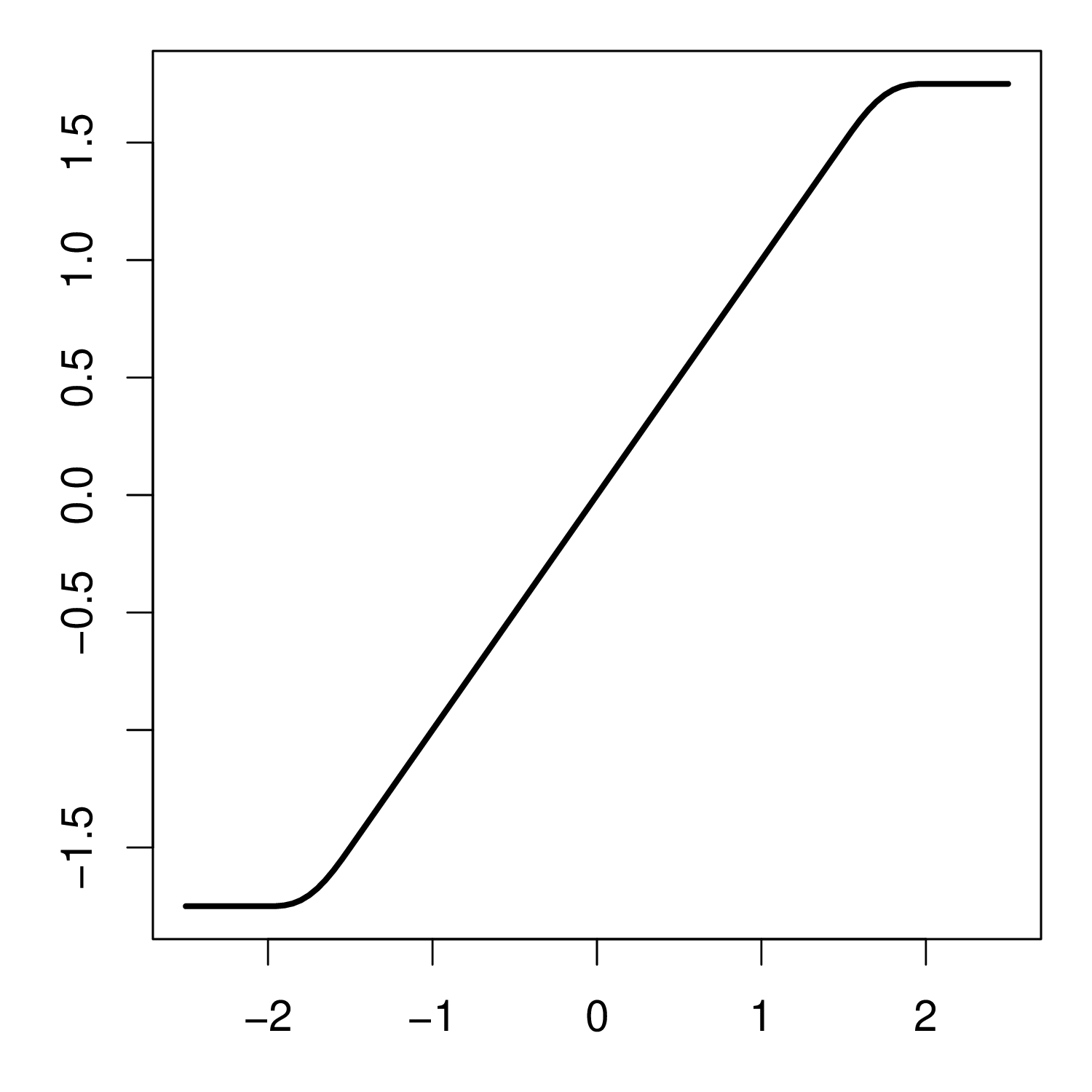}
\caption{Two times continuously differentiable version of Hubers $\Psi-$function used in the simulations.\label{hubersmooth}}
\end{center}
\end{figure}
\begin{table}
\begin{center}
\begin{tabular}{l|ccccc|ccccc}
&\multicolumn{5}{c|}{non-robust test}&\multicolumn{5}{|c}{robust test}\\
N T&50&100&200&400&800&50&100&200&400&800\\\hline
50&0.33&0.13&0.07&0.07&0.04&0.33&0.14&0.07&0.06&0.04\\
100&0.52&0.14&0.06&0.07&0.06&0.55&0.15&0.07&0.06&0.06\\
200&0.77&0.21&0.11&0.07&0.04&0.79&0.22&0.11&0.08&0.04\\
400&0.97&0.40&0.18&0.08&0.08&0.97&0.41&0.17&0.07&0.08\\
800&1.00&0.65&0.28&0.13&0.06&1.00&0.66&0.28&0.13&0.07
\end{tabular}
\caption{Empirical size under $\rho=0$, normal innovations and a significance level of 0.05.\label{H01}}
\end{center}
\end{table}
In case of a more heavy tailed distribution like a t-distribution with 3 degrees of freedom, results do not change much, as one can see in Table \ref{H02}. The non robust test seems to work well under the null hypothesis though its assumptions (finite 8-th moments) are violated.\\
\begin{table}[H]
\begin{center}
\begin{tabular}{l|ccccc|ccccc}
&\multicolumn{5}{c|}{non-robust test}&\multicolumn{5}{|c}{robust test}\\
N T&50&100&200&400&800&50&100&200&400&800\\\hline
50&0.30&0.11&0.05&0.08&0.06&0.34&0.10&0.05&0.08&0.06\\
100&0.47&0.15&0.08&0.06&0.06&0.54&0.16&0.08&0.06&0.06\\
200&0.74&0.19&0.10&0.06&0.05&0.77&0.25&0.11&0.07&0.06\\
400&0.97&0.36&0.12&0.08&0.06&0.97&0.41&0.15&0.09&0.06\\
800&1.00&0.65&0.23&0.14&0.08&1.00&0.68&0.26&0.15&0.08\\
\end{tabular}
\caption{Empirical size under $\rho=0,~t_3$ distributed innovations and a significance level of 0.05.\label{H02}}
\end{center}
\end{table}
Somewhat surprisingly the results get better under serial correlation ($\rho=0.5$) where we get reasonable empirical sizes already for $T=100.$ This can be partly explained by the relatively large bandwidth which better fits the case $\rho=0.5$ than $\rho=0.$
\begin{table}[H]
\begin{center}
\begin{tabular}{l|ccccc|ccccc}
&\multicolumn{5}{c|}{non-robust test}&\multicolumn{5}{|c}{robust test}\\
N T&50&100&200&400&800&50&100&200&400&800\\\hline
50&0.07&0.06&0.05&0.04&0.03&0.08&0.05&0.05&0.05&0.03\\
100&0.12&0.07&0.04&0.04&0.05&0.12&0.07&0.04&0.04&0.06\\
200&0.24&0.10&0.05&0.04&0.04&0.26&0.11&0.05&0.05&0.04\\
400&0.53&0.16&0.07&0.06&0.04&0.54&0.17&0.07&0.05&0.04\\
800&0.90&0.36&0.13&0.07&0.05&0.90&0.37&0.13&0.07&0.05
\end{tabular}
\caption{Empirical size under $\rho=0.5$, normal innovations and a significance level of 0.05.\label{H03}}
\end{center}
\end{table}
It is maybe preferable to use a data dependent bandwidth $b_T$ as proposed in \cite{andrews1991heteroskedasticity} or \cite{politis2011higher}. In this case one can either calculate a global bandwidth which is used for every individual or customized bandwidths $b_{T,i},~i=1,\ldots,N$. For the latter on has to ensure that the conditions (\ref{kernrate}) are also fulfilled for the infimum respectively the supremum of the bandwidths. However, automatic bandwidth choices are also questioned since they can produce very large bandwidths under the alternative of a level shift which considerably decrease the power of CUSUM tests. They can even lead to non monotonic power curves, see for example \cite{vogelsang1999sources} and \cite{crainiceanu2007nonmonotonic}.\\
Now we turn our focus at the behaviour under the alternative. Therefore we choose $\rho=0.25$, set $T=400,$ $N=200$ and add a jump at $t_0=200.$ Its height is generated for each individual independently by a normal distribution with mean 0 and standard deviation $\Delta$ varying between 0 and 0.2. Results for normal and t-distributed innovations with various degrees of freedom can be found in Figure \ref{H11}. We see that there is no visual difference between the tests under the alternative in case of normal data. This changes as the innovations get more heavy tailed. The robust test statistic performs superior in case of $t_5$ distributed innovations. This advantage increases if one looks at more heavy tailed distributions. Under $t_1$ innovations the power of the non robust test even does not exceed its size.\\
\begin{figure}[H]
\begin{center} 
\includegraphics[width=0.99\textwidth]{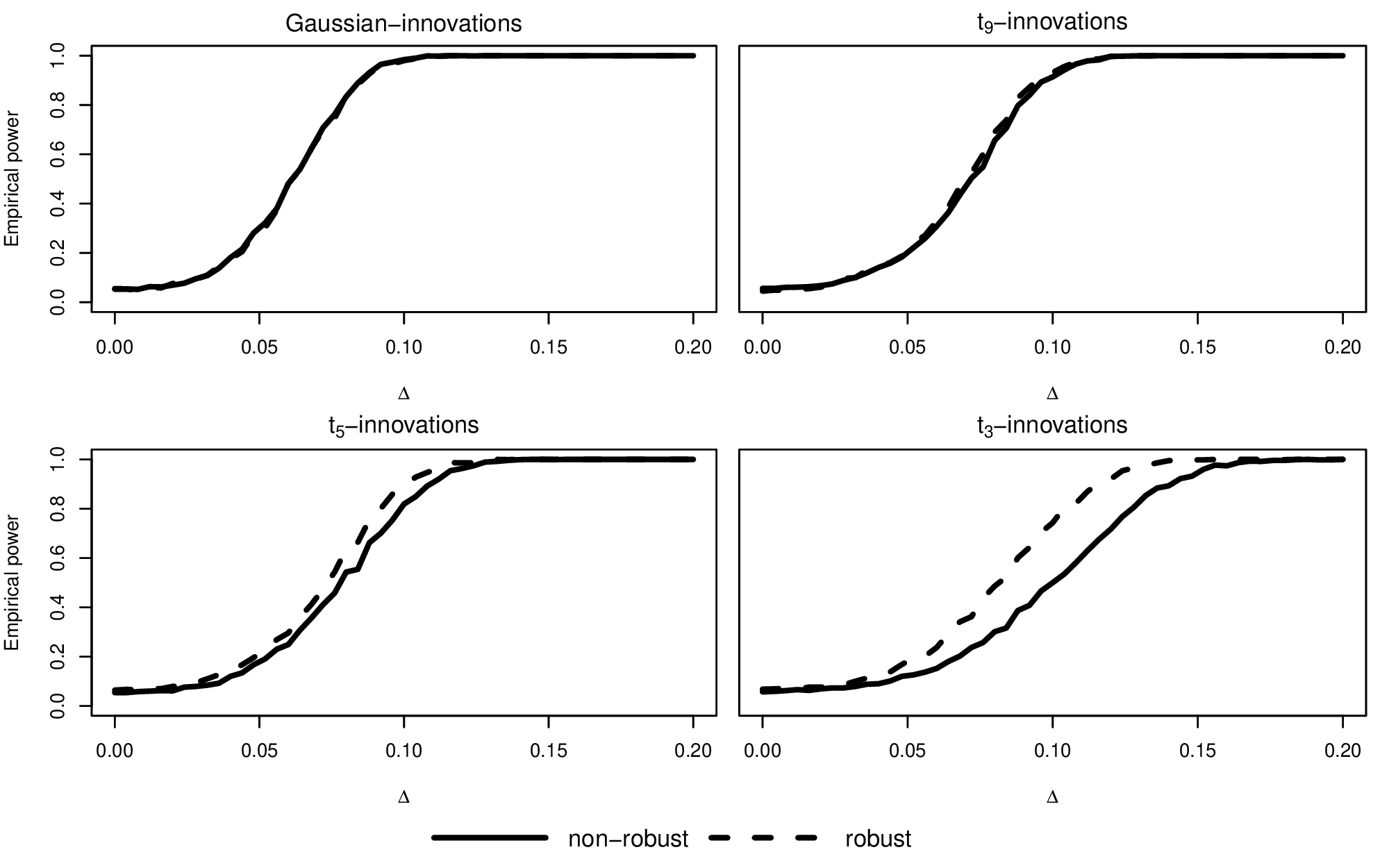}
\caption{Empirical power under $\rho=0.25,$ different distributed innovations, $N=200$, $T=400$ and a jump at $T=200$ which is generated by a normal distribution with mean 0 and standard deviation $\Delta$. \label{H11}}
\end{center}
\end{figure}

To sum up, the robust test performs comparably to the non robust one in the Gaussian case and clearly outperforms it under heavy tailed distributions.  
\section{Summary}
We have proposed a robust test for change-point detection in panel data where the number of individuals and the time horizon is large. The procedure is based on residuals which are robustly transformed via a $\Psi$-function. The null distribution is derived under very weak conditions, allowing for arbitrary heavy tailed distributions and heterogeneous serial dependence. To the best of our knowledge this is the first contribution in the robust change-point literature considering a data dependent choice of the tuning parameters $\sigma_i, i=1,\ldots,N,$ which allows to combine high robustness and efficiency.\\
A simulation study illustrates that the robust procedure outperforms the non robust one under heavy tails and moreover indicates that there is little loss under Gaussian data. This at first glance surprising observation is not unusual for robust methods applied in the high dimensional setting. For instance, it is observed that the Pitman asymptotic efficiency under Gaussianity of the spatial median approaches 1 if the dimension $N$ is large \citep{mottonen1997efficiency}. Asymptotic variances of robust scatter estimators like the Tyler shape matrix, MCD, M- and S-estimators tend to that of the empirical covariance matrix under multivariate normal distributions for large $N$ \citep[see e.g.][]{CrouxHaesbroeck1999, Taskinen2006}. Relative efficiencies of sign tests for uniformity on the unit sphere, independence against serial dependence and multivariate independence tend to 1 under normality and growing dimension compared to classical Gaussian competitors \citep[see e.g.][]{paindaveine2016high}. With regard to these results it is interesting to know whether one can prove similar results in the panel context.\\
There are also some limitations of our testing procedure. First there are no covariates (apart from an individual mean) in the considered model. The test statistic uses robustly transformed observations based on prior location- and scale-estimates. The intuitive generalisation would be to substitute regression residuals (from a robust regression) for the standardized residuals. However it is not clear under which conditions the asymptotic distribution is still the same as in the much simpler panel model considered here. Another limitation is the exclusion of cross sectional dependence. Promising in this regard looks the projection approach proposed by \cite{Aston}, which allows for an arbitrary dependence structure between the individuals. The main challenge here, also in the non-robust case, is the choice of the projection, since the power of the test crucially depends on it. Finally our test requires that the time dimension $T$ is large. It enables us to allow under some regularity conditions for arbitrary serial dependence which can even differ between the individuals. If $T$ is small one has to be more restrictive in this regard and it is also more complicated to allow for different distributions of the particular individuals.

\bibliographystyle{abbrvnat}
{\small
\bibliography{lit2}
}

\appendix
\section{Proofs}
The main component of our proofs are moment inequalities, which are based on coupling arguments. More in detail we use the following result by \cite{philipp1986invariance} which describes how dependent random variables can be substituted by independent ones.
\begin{proposition}\label{eq:philipp} (Theorem 3.4 in \cite{philipp1986invariance}): 
Let $\{B_k,m_k,k\geq 1\}$ be a sequence of Polish spaces. Let $\alpha_k$ denote the Borel field over $B_k,$ let $\{X_k,~k \geq 0\}$ be a sequence of random variables with values in $B_k$ and let $\{\gamma_k,~k\geq 0\}$ be a sequence of non-decreasing $\sigma-$fields such that $X_k$ is $\gamma_k-$measurable. Suppose that for some sequence $\{\beta_k\,~k\geq 0\}$ of non-negative numbers
\begin{align*}
\mathbb{E}\sup_{A\in \alpha_k}(|P(X_k\in A|\gamma_{k-1})-P(X_k\in A)|)\leq \beta_k
\end{align*}
for all $k\geq 1.$ Denote by $F_k$ the distribution of $X_k$ and let $\{G_k,~k\geq 0\}$ be a sequence of distributions on $(B_k,\alpha_k)$ such that
\begin{align*}
F_k(A)\leq G_k(A^{\rho_k})+\sigma_k~~\forall~A\in \alpha_k
\end{align*}
with $\rho_k,~\sigma_k\geq0$ and $A^{\epsilon}=\cup_{x\in A}\{y:~m_k(x,y)\leq \epsilon\}.$ Then without changing its distribution, one can redefine the sequence $\{X_k,~k\geq 0\}$ on a richer probability space on which there exists a sequence $\{Y_k,~k\geq 1\}$ of independent random variables with
 distribution $G_k$ such that for all $k\geq 1$:
 \begin{align*}
 P(m_k(X_k,Y_k)>\rho_k)\leq \beta_k+\sigma_k.
 \end{align*}
 \end{proposition}
We use Proposition \ref{eq:philipp} to derive covariance inequalities for processes which are P-NED. Similar inequalities are already proofed for processes which are strong mixing \citep[see][]{davydov1970invariance,deo1973note} or $L_1$-NED \citep[see][]{borovkova2001limit}.
\begin{proposition}\label{eq:eqmoment}
Let $(X_t)_{t \in \mathbb{N}}$ be stationary and P-NED on an absolutely regular process $(Z_{t})_{t\in \mathbb{Z}}$ with approximation constants $(a_k)_{k\in \mathbb{N}}$, functions $(f_k)_{k\in \mathbb{N}}$, error function $\Phi$ and absolutely regularity coefficients $(\beta_k)_{k\in \mathbb{N}}.$ Furthermore there exists $a>p\geq 2$ with $\mathbb{E}(|X_1|^a)<\infty$, \linebreak $\int_0^1 x^{\frac{a}{a-p+1}-1}\Phi(x) dx < \infty$, $\int_1^\infty x^{a-1} \Phi(x) dx<\infty,$ then there are $D_1$ and $D_2$ independent of $m$ such that
\begin{align}\label{covin}
|\mathbb{E}(X_{i_1}\ldots X_{i_k}X_{i_{k+1}}\ldots X_{i_p})-\mathbb{E}(X_{i_1}\ldots X_{i_k})\mathbb{E}(X_{i_{k+1}}\ldots X_{i_p})|\leq D_1 a_{\lfloor m/3 \rfloor}^{\frac{a-p+1}{a}}+D_2\beta_{\lfloor m/3 \rfloor}^{\frac{a-p}{p}}
\end{align}
where $1\leq i_1 \leq \ldots \leq i_p\leq T$ and $m=i_{k+1}-i_k.$
\end{proposition}
\begin{proof}[Proof of Proposition \ref{eq:eqmoment}]
The proof follows the ideas of \cite{borovkova2001limit}.
First we build independent blocks $\tilde{W}_1=\tilde{Z}_{i_1-\lfloor m/3\rfloor},\ldots, \tilde{Z}_{i_k+\lfloor m/3\rfloor}$ and $\tilde{W}_2=\tilde{Z}_{i_{k+1}-\lfloor m/3\rfloor},\ldots, \tilde{Z}_{i_{p}+\lfloor m/3\rfloor}$ where the functions $f_m$ will work. Denote the original blocks by $W_1=Z_{i_1-\lfloor m/3\rfloor},\ldots, Z_{i_k+\lfloor m/3\rfloor}$ and $W_2=Z_{i_{k+1}-\lfloor m/3\rfloor},\ldots, Z_{i_p+\lfloor m/3\rfloor}.$ We apply Proposition \ref{eq:philipp} with $X_1=W_1$ and $X_2=W_2$ as well as $\rho_k=\sigma_k=0$ for $k=1,2$. Then $B_1=\mathbb{R}^{i_k-i_1+2\lfloor m/3 \rfloor+1} $ and $B_2=\mathbb{R}^{i_{p}-i_{k+1}+2\lfloor m/3 \rfloor+1} $ are polish spaces, $\gamma_1=\sigma(\ldots,Z_{i_k+\lfloor m/3 \rfloor}),$ $\gamma_2=\sigma(\ldots,Z_{i_p+\lfloor m/3 \rfloor})$ and 
\begin{align*}
\mathbb{E}\sup_{A\in \alpha_r}(|P(W_r\in A|\gamma_{r-1})-P(W_r\in A)|)\leq \beta_{m-2\lfloor m/3\rfloor}\leq \beta_{\lfloor m/3 \rfloor} ~r=1,2
\end{align*}
since both blocks are separated by $m-2\lfloor m/3\rfloor$. Proposition \ref{eq:philipp} then guarantees the existence of independent blocks $\tilde{W}_1$ and $\tilde{W}_2$ which are distributed as $W_1$ and $W_2$ such that:
\begin{align*} P(\tilde{W}_i\neq W_i)\leq \beta_{\lfloor m/3 \rfloor}~\mbox{for}~i=1,2
\end{align*}
which entails that the corresponding functionals $f_{\lfloor m/3 \rfloor}(\tilde{Z}_{l-\lfloor m/3\rfloor},\ldots,\tilde{Z}_{l+\lfloor m/3 \rfloor})=:\tilde{X}_l$ and \linebreak $f_{\lfloor m/3 \rfloor}(\tilde{Z}_{n-\lfloor m/3\rfloor},\ldots,\tilde{Z}_{n+\lfloor m/3 \rfloor})=:\tilde{X}_
n$ are also independent as long as $l<=i_k$ and $n\geq i_{k+1}$. In the following we need a bound for the error between $\tilde{X}_l$ and $X_l$ for $l\in {i_1,\ldots,i_p}.$ For $i_1\leq l\leq i_k$ one gets the following decomposition of the error
\begin{align}\label{eq:eindif}
\mathbb{E}(|X_{l}-\tilde{X}_{l}|^{\frac{a}{a-p+1}})&=\mathbb{E}(|(X_{l}-\tilde{X}_{l})I_{W_1=\tilde{W}_1}|^{\frac{a}{a-p+1}})+\mathbb{E}(|(X_{l}-\tilde{X}_{l})I_{W_1\neq \tilde{W}_1}|^{\frac{a}{a-p+1}}).
\end{align}
For the first term one can use the P-NED property iii)
\begin{align*}
\mathbb{E}(|(X_{l}-\tilde{X}_{l})I_{W_1=\tilde{W}_1}|^{\frac{a}{a-p+1}} )\leq a_{\lfloor m/3 \rfloor}\int_0^\infty \epsilon^{\frac{a}{a-p+1}-1} \Phi(\epsilon) d \epsilon \leq a_{\lfloor m/3 \rfloor} \cdot C_1
\end{align*}
where $C_1$ only depends on $\Phi.$ Using the $c_r$ inequality \citep*[see][p. 157]{Loeve1977} and H\"{o}lder inequality we get
\begin{align*}
\mathbb{E}&(|(X_{l}-\tilde{X}_{l})I_{W_1\neq \tilde{W}_1}|^{\frac{a}{a-p+1}})\\
&\leq 2^{\frac{a}{a-p+1}}\left(\mathbb{E}(|X_{l}I_{W_1\neq \tilde{W}_1}|^{\frac{a}{a-p+1}})+\mathbb{E}(|\tilde{X}_{l}I_{W_1\neq \tilde{W}_1}|^{\frac{a}{a-p+1}})\right)\nonumber\\
&\leq 2^{\frac{a}{a-p+1}}\left([\mathbb{E}(|X_l|^a)]^{\frac{1}{a-p+1}}P(W_1\neq \tilde{W}_1)^{\frac{a-p}{a-p+1}}+[\mathbb{E}(|\tilde{X}_l|^a)]^{\frac{1}{a-p+1}}P(W_1\neq \tilde{W}_1)^{\frac{a-p}{a-p+1}}\right).\nonumber\\
& \leq 2^{\frac{a}{a-p+1}} \beta_{\lfloor m/3 \rfloor}^{\frac{a-p}{a-p+1}}\left([\mathbb{E}(|X_l|^a)]^{\frac{1}{a-p+1}}+[\mathbb{E}(|\tilde{X}_l|^a)]^{\frac{1}{a-p+1}}\right)\nonumber
\end{align*} 
and since $\tilde{W}_1$ is distributed as $W_1$ we have
\begin{align}\label{eq:helf2}
\mathbb{E}(|\tilde{X}_l|^a)&=\mathbb{E}(|f_{\lfloor m/3 \rfloor}(Z_{l-\lfloor m/3 \rfloor},\ldots,Z_{l-\lfloor m/3 \rfloor})|^a)\nonumber \\
&\leq 2^{a-1} (\mathbb{E}(|f(Z_{l-\lfloor m/3 \rfloor},\ldots,Z_{l-\lfloor m/3 \rfloor})-X_l|^a)+\mathbb{E}(|X_l|^a)\nonumber\\
&\leq 2^{a-1} (a_{\lfloor m/3 \rfloor} C_1+\mathbb{E}(|X_l|^a)).
\end{align}
Therefore (\ref{eq:eindif}) is bounded by
\begin{align}\label{eq:helf}
\mathbb{E}(|X_{l}-\tilde{X}_{l}|^{\frac{a}{a-p+1}}) \leq C_2(a_{\lfloor m/3 \rfloor}+\beta_{\lfloor m/3 \rfloor}^{\frac{a-p}{a-p+1}})
\end{align}
which also holds for $i_{k+1}\leq l \leq i_p.$
To derive covariance inequalities, we need a bound for the error between products of random variables and its copies
\begin{align}\label{eq:sum1}
\mathbb{E}&(|X_{i_1} \ldots X_{i_p}-\tilde{X}_{i_1} \ldots \tilde{X}_{i_p}|)\\
&\leq\mathbb{E}(|(X_{i_1}-\tilde{X}_{i_1})(X_{i_2}\ldots X_{i_p})|)+\sum_{j=2}^{p-1} \mathbb{E}(| \tilde{X}_{i_1}\ldots \tilde{X}_{i_{j-1}})(X_{i_j}-\tilde{X}_{i_j})(X_{i_{j+1}}X_{i_p}|)\nonumber \\
&+\mathbb{E}(|(\tilde{X}_{i_1}\ldots \tilde{X}_{i_{p-1}})(X_{i_p}-\tilde{X}_{i_p})|).\nonumber
\end{align}
Using the generalized H\"older inequality and (\ref{eq:helf}) one can bound the first summand on the right hand-side of (\ref{eq:sum1}) by
\begin{align*}
\mathbb{E}&(|(X_{i_1}-\tilde{X}_{i_1})(X_{i_2}\ldots X_{i_p})|)\\
&\leq \left[\mathbb{E}\left(|(X_{i_1}-\tilde{X}_{i_1})|^{\frac{a}{a-p+1}}\right)\right]^{\frac{a-p+1}{a}}\left[\mathbb{E}\left(|X_{i_2}|^a\right)\right]^\frac{1}{a}\cdots \left[\mathbb{E}\left(|X_{i_p}|^a\right)\right]^\frac{1}{a}\\
&\leq C_4(a_{\lfloor m/3 \rfloor}+\beta_{\lfloor m/3 \rfloor}^{\frac{a-p}{a-p+1}})^\frac{a-p+1}{a}.
\end{align*}
Similar bounds for the other summands in (\ref{eq:sum1}) can be derived using (\ref{eq:helf}) and (\ref{eq:helf2}) which results eventually in
\begin{align}\label{eq:bound1}
\mathbb{E}&(|X_{i_1} \ldots X_{i_p}-\tilde{X}_{i_1} \ldots \tilde{X}_{i_p}|) \leq C_5 a_{\lfloor m/3 \rfloor}^\frac{a-p+1}{a}+C_6\beta_{\lfloor m/3 \rfloor}^\frac{a-p}{a}.
\end{align}
Analogously one gets
\begin{align}\label{eq:bound2}
\mathbb{E}&(|X_{i_1}\ldots X_{i_k}-\tilde{X}_{i_1}\ldots \tilde{X}_{i_k}|)\leq C_7 a_{\lfloor m/3 \rfloor}^\frac{a-k+1}{a}+C_8\beta_{\lfloor m/3 \rfloor}^\frac{a-k}{a}
\end{align}
and
\begin{align}\label{eq:bound3}
\mathbb{E}&(|X_{i_1}\ldots X_{i_k}-\tilde{X}_{i_1}\ldots \tilde{X}_{i_k}|)\leq C_9 a_{\lfloor m/3 \rfloor}^\frac{a-p+k+1}{a}+C_{10}\beta_{\lfloor m/3 \rfloor}^\frac{a-p+k}{a}
\end{align}
Finally we prove the covariance inequality \ref{covin}. Denote
$A=X_{i_1}\ldots X_{i_j}$, $B=X_{i_{j+1}} \ldots X_{i_p}$, $\tilde{A}=\tilde{X}_{i_1}\ldots \tilde{X}_{i_k}$ and $\tilde{B}=\tilde{X}_{i_{k+1}}\ldots \tilde{X}_{i_p}$ then we get by (\ref{eq:bound1}), (\ref{eq:bound2}) and (\ref{eq:bound3})
\begin{align*}
|\mathbb{E}(AB)-\mathbb{E}(A)\mathbb{E}(B)|&=|\mathbb{E}(AB)-\mathbb{E}(A-\tilde{A}+\tilde{A})\mathbb{E}(B-\tilde{B}+\tilde{B})|\\
&\leq \mathbb{E}(|AB-\tilde{A}\tilde{B}|)+\mathbb{E}(|\tilde{A}|)\mathbb{E}(|B-\tilde{B}|)\\
&+\mathbb{E}(|\tilde{B}|)\mathbb{E}(|A-\tilde{A}|)+\mathbb{E}(|A-\tilde{A}|)\mathbb{E}(|B-\tilde{B}|)\\
&\leq C_{11} a_{\lfloor m/3 \rfloor}^\frac{a-p+1}{a}+C_{12}\beta_{\lfloor m/3 \rfloor}^\frac{a-p}{a}
\end{align*}
which completes the proof.\end{proof}
The result is a little sharper if the process is bounded.
\begin{proposition}\label{eq:eqmomentbounded}
Let $(X_t)_{t \in \mathbb{N}}$ be stationary and P-NED on an absolutely regular process $(Z_{t})_{t\in \mathbb{Z}}$ with approximation constants $(a_k)_{k\in \mathbb{N}}$, functions $(f_k)_{k\in \mathbb{N}}$, error function $\Phi$ and absolutely regularity coefficients $(\beta_k)_{k\in \mathbb{N}}.$ Furthermore there exists $D$ such that $|X_1|\leq D$ a.s. and $\int_0^\infty \Phi(x)dx<\infty$, then there exists $D_1$ independent of $m$ such that
\begin{align*}
|\mathbb{E}(X_{i_1}\ldots X_{i_k}X_{i_{k+1}}\ldots X_{i_p})-\mathbb{E}(X_{i_1}\ldots X_{i_k})\mathbb{E}(X_{i_{k+1}}\ldots X_{i_p})|\leq D_1 (a_{\lfloor m/3 \rfloor}+\beta_{\lfloor m/3 \rfloor}),
\end{align*}
where $1\leq i_1 \leq \ldots \leq i_p\leq T$ and $m=i_{k+1}-i_k.$
\end{proposition}
The proof is analogous to that of Proposition \ref{eq:eqmoment}. But instead of using the H\"older inequality one uses the boundedness of $X_i$, which enables us to extract the largest possible absolute value $D$ out of the expectation.\\
The next theorem covers the case where one looks at transformations of the P-NED process. It is necessary for the consistency of the long run variance estimation under estimated tuning parameters.
\begin{proposition}\label{eq:eqmomentboundeddifferent}
Let $(X_t)_{t \in \mathbb{N}}$ be stationary and P-NED on an absolutely regular process $(Z_{t})_{t\in \mathbb{Z}}$ with approximation constants $(a_k)_{k\in \mathbb{N}}$, functions $(f_k)_{k\in \mathbb{N}}$, absolutely regularity coefficients $(\beta_k)_{k\in \mathbb{N}}$ and error function $\Phi$ with $\int_0^\infty \Phi(x)dx<\infty$. Furthermore let $g_1,\ldots,g_p:\mathbb{R}\rightarrow \mathbb{R}$ be bounded and Lipschitz continuous functions with universal Lipschitz constant $c$, then there exists $D_1$ independent of $m$ such that
\begin{align*}
|\mathbb{E}[g_1(X_{i_1})& \ldots g_k(X_{i_k})g_{k+1}(X_{i_{k+1}})\ldots g_p(X_{i_p})]\\
&-\mathbb{E}[g_1(X_{i_1})\ldots g_k(X_{i_k})]\mathbb{E}[g_{k+1}(X_{i_{k+1}})\ldots g_p(X_{i_p})]|\leq D_1 (a_{\lfloor m/3 \rfloor}+\beta_{\lfloor m/3 \rfloor}),
\end{align*}
where $1\leq i_1 \leq \ldots \leq i_p\leq T$ and $m=i_{k+1}-i_k.$
\end{proposition}
The proof is completely analogous to that of Proposition \ref{eq:eqmomentbounded}.\\
The next proposition is a Marcinkiewicz-Zygmund type inequality which bounds the $p-$th moment of the sum of the process. 
\begin{proposition}\label{eq:Bernstein}
Let $(X_t)_{t \in \mathbb{N}}$ be stationary and P-NED on an absolutely regular process $(Z_{t})_{t\in \mathbb{Z}}$ with approximation constants $(a_k)_{k\in \mathbb{N}}$, functions $(f_k)_{k\in \mathbb{N}}$, error function $\Phi$ and regularity coefficients $(\beta_k)_{k\in \mathbb{N}}.$ Furthermore there exists $a>p\in  \mathbb{N}$ such that $\mathbb{E}(|X_1|^a)<\infty$, $\int_0^1 x^{\frac{a}{a-p+1}-1}\Phi(x)dx<\infty$, $\int_1^\infty x^{a-1} \Phi(x) dx<\infty$ and $\sum_{i=1}^\infty a_i^{\frac{a-p+1}{a}} i^{p-1}$ as well as $\sum_{i=1}^\infty \beta_i^{\frac{a-p}{a}} i^{p-1}<\infty$, then there exist $G_1$ and $G_2$ such that
\begin{align}\label{eq:indup}
|\mathbb{E}(\sum_{i=1}^T X_i)^p|\leq TG_1|\mathbb{E}(X_1)|^p+G_2T^{\lfloor p/2 \rfloor},~ \forall T\in \mathbb{N}.
\end{align}
\end{proposition}
\begin{proof}[Proof of Proposition \ref{eq:Bernstein}]
First notice that
\begin{align}\label{eq:trivial}
|\mathbb{E}(\sum_{i=1}^T X_i)^p|\leq \sum_{i_1,\ldots,i_p=1}^T|\mathbb{E}X_{i_1}\ldots X_{i_p}|.
\end{align}
We actually show the result by induction applied to the right hand side of (\ref{eq:trivial}). For $p=1$ the right hand side of (\ref{eq:trivial}) is obviously bounded by the right hand side of (\ref{eq:indup}). For the induction step $p\rightarrow p+1$ we want to split the expectations where the time difference is largest. Ordering of time indices yields
\begin{align}\label{eq:zwi}
\sum_{i_1,\ldots,i_{p+1}=1}^T&|\mathbb{E}(X_{i_1}\ldots X_{i_{p+1}})|=\sum_{1\leq i_1\leq \ldots \leq i_{p+1}}^T|\mathbb{E}(X_{i_1}\ldots X_{i_{p+1}})\gamma(i_1,\ldots,i_{p+1})|
\end{align}
where $\gamma(i_1,\ldots,i_{p+1})$ denotes the number of possible permutations which is smaller or equal $(p+1)!$. Let $j_s=i_s-i_{s-1},~s=2,\ldots,p+1$ and $j_1=i_1$, then (\ref{eq:zwi}) is bounded by
\begin{align}\label{grsum}
(p+1)!\sum_{j_1,\ldots,j_{p+1}\geq 0~ j_1+\ldots+j_{p+1}\leq T}^T|\mathbb{E}(X_{j_1}X_{j_1+j_2}\ldots X_{j_1+\ldots +j_{p+1}})|.
\end{align}
We divide the sum (\ref{grsum}) into $A_1,\ldots,A_p$ where $A_s$ contains all expectations in (\ref{grsum}) where $j_{s+1}$ is the maximum of $j_2,\ldots,j_{p+1}$\footnote{To obtain a unique partition we add summands to the sum with the smallest index, if the maximum is attained more than once} and denote $I_s$ the related index set. If $j_{s+1}$ is maximal the other indices can only assume the values $0,\ldots,j_{s+1}$ resulting in less than $(p+1)!T(j_{s+1}+1)^{p-1}$ summands for fixed $j_{s+1}.$ Proposition \ref{eq:eqmoment} yields
\begin{align}\label{divex}
|A_s|&=(p+1)!\sum_{i_1,\ldots,i_{p+1}\in I_s}|\mathbb{E}(X_{i_1}\ldots X_{i_{p+1}})|\nonumber\\
&\leq (p+1)!\sum_{i_1,\ldots,i_{p+1}\in I_s}|\mathbb{E}(X_{i_1}\ldots X_{i_{s-1}})\mathbb{E}(
X_{i_s}\ldots X_{i_{p+1}})|\\
&+(p+1)!T \sum_{j_s=0}^T \left(C_{13} a_{\lfloor j_s/3\rfloor}^{\frac{a-p+1}{a}} +C_{14}\beta_{\lfloor j_s/3\rfloor}^{\frac{a-p}{p}}\right)(j_s+1)^{p-1}\label{divex2}
\end{align}
where the sum (\ref{divex2}) is $O(T)$ by assumption. We use the induction hypothesis for (\ref{divex}) to obtain
\begin{align*}
(p+1)!&\sum_{i_1,\ldots,i_{s-1}=1}^T |\mathbb{E}(X_{i_1}\ldots X_{i_{s-1}})|\sum_{i_s,\ldots,i_{p+1}=1}^T|\mathbb{E}(X_{i_s}\ldots X_{i_{p+1}})|\\
&\leq (p+1)! (G_1T^{s-1}|\mathbb{E}(X_1)|^{s-1}+G_2T^{\lfloor \frac{s-1}{2} \rfloor})(\tilde{G}_1T^{p+2-s}|\mathbb{E}(X_1)|^{p+2-s}+\tilde{G}_2T^{\lfloor \frac{p+2-s}{2} \rfloor})\\
&\leq(p+1)!\left(G_1\tilde{G}_1T^{p+1}|\mathbb{E}(X_1)|^{p+
1}+G_2\tilde{G}_2 T^{\lfloor \frac{p+1}{2} \rfloor}\right.\\
&\left.+G_1\tilde{G}_2T^{\lfloor \frac{p+2-s}{2}\rfloor}T^{s-1}|\mathbb{E}(X_1)|^{s-1}+\tilde{G}_1G_2T^{\lfloor \frac{s-1}{2} \rfloor}T^{p+2-s}|\mathbb{E}(X_1)|^{p+2-s}\right)\\
&\leq \hat{G}_1 T^{p+1}|\mathbb{E}(X_1)|^{p+1}+\hat{G}_2 T^{\lfloor \frac{p+1}{2} \rfloor}
\end{align*}
which completes the proof.\end{proof}
There is also a version for bounded processes.
\begin{proposition}\label{eq:Bernsteinbound}
Let $(X_t)_{t \in \mathbb{N}}$ be stationary and P-NED on an absolutely regular process $(Z_{t})_{t\in \mathbb{Z}}$ with approximation constants $(a_k)_{k\in \mathbb{N}}$, functions $(f_k)_{k\in \mathbb{N}}$, error function $\Phi$ and absolutely regularity coefficients $(\beta_k)_{k\in \mathbb{N}}.$ Furthermore there exists $D\in \mathbb{R}$ such that $|X_1|<D$ a.s., $\int_0^\infty \Phi(x)dx<\infty$ and $\sum_{i=1}^\infty a_i i^{p-1}$ as well as $\sum_{i=1}^\infty \beta_i i^{p-1}<\infty$, then there exist $G_1$ and $G_2$such that
\begin{align}\label{eq:indup2}
|\mathbb{E}(\sum_{i=1}^N X_i)^p|\leq G_1(T|\mathbb{E}(X_1)|)^p+G_2T^{\lfloor p/2 \rfloor},~\forall T\in \mathbb{N}.
\end{align}
\end{proposition}
The proof is completely analogous to that of Proposition \ref{eq:Bernstein} using Proposition
\ref{eq:eqmomentbounded} instead of Proposition \ref{eq:eqmoment}.\\

The proof of Theorem \ref{T1} consists of four steps
\begin{enumerate}
\item the mean function of $(W_{N,T}(x))_{x\in [0,1]}$ converges to that of $(\Gamma(x))_{x\in [0,1]}$,
\item the covariance function of $(W_{N,T}(x))_{x\in [0,1]}$ converges to that of $(\Gamma(x))_{x\in [0,1]}$,
\item all finite dimensional distributions of $(W_{N,T}(x))_{x\in [0,1]}$ converge against a multivariate normal distribution,
\item $(W_{N,T}(x))_{x\in [0,1]}$ is tight.
\end{enumerate} 
Let without loss of generality $\mathbb{E}(Y_{i,1})=0$ and denote $\gamma_{i}(h)=\mathbb{E}(Y_{i,t}Y_{i,t+h})$ the autocovariance of lag $h\in \mathbb{N}.$ Furthermore we want to concentrate on the bounded case $|Y_{i,1}|<c_1,~i=1\ldots,N.$ The proof in the unbounded case is completely analogous. However, the covariance inequalities then also depend on $a$, making calculations a little more extensive and harder to understand.  
\begin{proof}[Step 1 of the proof of Theorem \ref{T1}]
First we look at one individual $i.$ Let $x,~0<x<1,$ be arbitrary and $k=\lfloor Tx\rfloor$
\begin{align*}
\mathbb{E}([S_T^{(i)}(x)]^2)&=\frac{1}{Tv_i^2}\mathbb{E}\left(\left[\frac{T-k}{T}\sum_{t=1}^{k}Y_{i,t}-\frac{k}{T}\sum_{t=k+1}^TY_{i,t}\right]^2\right)\\
&=\frac{1}{Tv_i^2}\left(\frac{T-k}{T}\right)^2\underbrace{\mathbb{E}\left(\left[\sum_{t=1}^kY_{i,t}\right]^2\right)}_{A1}\\
&-\frac{1}{Tv_i^2}2\frac{(T-k)k}{T^2}\underbrace{\mathbb{E}\left(\left[\sum_{t=1}^kY_{i,t}\right)\left(\sum_{t=k+1}^TY_{i,t}\right]\right)}_{A_2}
+\frac{1}{Tv_i^2}\frac{k^2}{T^2}\underbrace{\mathbb{E}\left(\left[\sum_{t=k+1}^TY_{i,t}\right)^2\right]}_{A3}.
\end{align*}
Elementary calculations yield
\begin{align*}
A_1&=kv_i^2-2k\sum_{h=k}^\infty \gamma_i(h)-2\sum_{h=1}^{k-1}h\gamma_i(h)\\
A_2&=\sum_{h=1}^T\gamma_i(h)\min(h,k,T-k,T-h)\\
A_3&=(T-k)v_i^2-2(T-k)\sum_{h=T-k}^\infty \gamma_i(h)-2\sum_{h=1}^{T-k-1}h\gamma_i(h)
\end{align*}
and therefore
\begin{align*}
\mathbb{E}(W_{N,T}(x))&=\frac{\sqrt{N}}{T}\frac{1}{N}\sum_{i=1}^N \left(\frac{1}{v_i^2}\left(\frac{T-k}{T}\right)^2\left(-2k\sum_{h=k}^\infty \gamma_i(h)-2\sum_{h=1}^{k-1}h\gamma_i(h)\right)\right.\\
&+\frac{1}{v_i^2}\frac{(T-k)k}{T^2}\sum_{h=1}^T\gamma_i(h)\min(h,k,T-k,T-h)\\
&\left.+\frac{1}{v_i^2}\frac{k^2}{T^2}
\left(-2(T-k)\sum_{h=T-k}^\infty \gamma_i(h)-2\sum_{h=1}^{T-k-1}h\gamma_i(h)\right)\right).
\end{align*}
Using Assumption 1 2) a) ii) one has:
\begin{align*}
|\mathbb{E}(W_{N,T}(x))|&\leq \frac{\sqrt{N}}{T}\frac{1}{N}\sum_{i=1}^N \frac{C_{15}}{\delta^2}\left( 5\underbrace{\sum_{h=1}^\infty h (1+h)^{-b}}_{B_1}+2\underbrace{k\sum_{h=k}^\infty (1+h)^{-b}}_{B_2}+2\underbrace{(T-k)\sum_{h=T-k}^\infty (1+h)^{-b}}_{B_3}\right)
\end{align*}
where by the integral criteria for sums
\begin{align*}
B_1\leq \int_0^\infty (1+h)^{-b+1} dh=\frac{1}{b-2}\mbox{~and~}
B_2\leq k \int_{k-1}^\infty c(1+h)^{-b} dh=\frac{k^{-b+2}}{b-1}\leq \frac{1}{b-1}.
\end{align*}
and analogously $B_3\leq \frac{(T-k)^{-b+2}}{b-1}$.
By Assumption i) of Theorem \ref{T1} $|\mathbb{E}(W_{N,T}(x))|\rightarrow 0$ which gives the desired mean structure.
\end{proof}
\begin{proof}[Step 2 of the proof of Theorem \ref{T1}]
Calculating the covariance structure is more tedious. By Assumption 1 1) we have
\begin{align*}
\mbox{Cov}(W_{N,T}(x),W_{N,T}(y))=\frac{1}{N}\sum_{i=1}^N Cov([S^{(i)}(x)]^2,[S^{(i)}(y)]^2),
\end{align*}
so it is enough to compute the covariance structure of one individual $i,~i=1,\ldots,N.$
We denote therefore $k=[xn] < [yn]=j$ and expand $E([S^{(i)}(x)]^2[S^{(i)}(y)]^2)$ to obtain
\begin{align*}E([S^{(i)}(x)]^2[S^{(i)}(y)]^2)&=\frac{(T-k)^2(T-j)^2}{T^6v_i^4}\mathbb{E}\left(\sum_{s,t,u,v=1}^{k,k,j,j}Y_{i,t}Y_{i,s}Y_{i,u}Y_{i,v}\right)\\
&-2\frac{(T-k)^2(T-j)j}{T^6v_i^4}\mathbb{E}\left(\sum_{s,t,u=1,v={j+1}}^{k,k,j,T}Y_{i,s}Y_{i,t}Y_{i,u}Y_{i,v}\right)
\\
&+\frac{(T-k)^2j^2}{T^6v_i^4}\mathbb{E}\left(\sum_{s,t=1,u,v=j+1}^{k,k,T,T}Y_{i,s}Y_{i,t}Y_{i,u}Y_{i,v}\right)
\\
&-2\frac{(T-k)k(T-j)^2}{T^6v_i^4}\mathbb{E}\left(\sum_{s=1,t=k+1,u,v=1}^{k,T,j,j}Y_{i,s}Y_{i,t}Y_{i,u}Y_{i,v}\right)
\\
&+4\frac{(T-k)k(T-j)j}{T^6v_i^4}\mathbb{E}\left(\sum_{s=1,t=k+1,u=1,v={j+1}}^{k,T,j,T}Y_{i,s}Y_{i,t}Y_{i,u}Y_{i,v}\right)
\\
&-2\frac{(T-k)kj^2}{T^6v_i^4}\mathbb{E}\left(\sum_{s=1,t=k+1,u,v={j+1}}^{k,T,T,T}Y_{i,s}Y_{i,t}Y_{i,u}Y_{i,v}\right)\\
&+\frac{k^2(T-j)^2}{T^6v_i^4}\mathbb{E}\left(\sum_{s,t=k+1,u,v=1}^{T,T,j,j}Y_{i,s}Y_{i,t}Y_{i,u}Y_{i,v}\right)\\
&-2\frac{k^2(T-j)j}{T^6v_i^4}\mathbb{E}\left(\sum_{s,t=k+1,u=1,v={j+1}}^{T,T,j,T}Y_{i,s}Y_{i,t}Y_{i,u}Y_{i,v}\right)
\\
&+\frac{k^2j^2}{T^6v_i^4}\mathbb{E}\left(\sum_{s,t=k+1,u,v={j+1}}^{T,T,T,T}Y_{i,s}Y_{i,t}Y_{i,u}Y_{i,v}\right)
\\
&=A_{i,1}+A_{i,2}+A_{i,3}+A_{i,4}+A_{i,5}+A_{i,6}+A_{i,7}+A_{i,8}+A_{i,9}.
\end{align*}
Exemplarily we look at $A_{i,1}$ which we divide into
\begin{align*}
A_{i,1}&=\frac{(T-k)^2(T-j)^2}{T^6v_i^4}\left[\mathbb{E}\left(\sum_{s,t,u,v=1}^{k,k,k,k}Y_{i,t}Y_{i,s}Y_{i,u}Y_{i,v}\right)+\mathbb{E}\left(\sum_{s,t=1,u,v=k+1}^{k,k,j,j}Y_{i,t}Y_{i,s}Y_{i,u}Y_{i,v}\right)\right]\\
&=\frac{(T-k)^2(T-j)^2}{T^6v_i^4}(B_{i,1}+B_{i,2}).
\end{align*}
We split $B_{1,i}$ further in parts which are substantial and parts which are negligible. Ordering such that $t\leq s\leq u\leq v$ and substituting $s-t=l,~u-s=m,~v-u=n$ yields
\begin{align*}
B_{1,i}=\sum_{t+l+m+n\leq k}\mathbb{E}\left(Y_{i,t}
Y_{i,t+s}Y_{i,t+s+m}Y_{i,t+s+m+n}\right) a(s,m,n)
\end{align*}
where $a(s,m,n)$ equals the number of permutations of $Y_{i,t}Y_{i,s+t}Y_{i,t+s+m}Y_{i,t+s+m+n}$. Now we want to apply Proposition \ref{eq:eqmomentbounded} and split the expectations where the lag difference between the random variables is largest: 
\begin{align*}
B_{1,i}
&=\sum_{t+s+m+n\leq k,~m,n< s}\mathbb{E}(Y_{i,t})\mathbb{E}(Y_{i,t+s}Y_{i,t+s+m}Y_{i,t+s+m+n})a(s,m,n)+R_{1,i}\\\
&+\sum_{t+s+m+n\leq k,~ s,n \leq m}\mathbb{E}(Y_{i,t}Y_{i,t+s})\mathbb{E}(Y_{i,t+s+m}Y_{i,t+s+m+n})a(s,m,n)+R_{2,i}\\
&+\sum_{t+s+m+n\leq k,~s,m< n}\mathbb{E}(Y_{i,t}Y_{i,t+s}Y_{i,t+s+m})\mathbb{E}(Y_{i,t+s+m+n})a(s,m,n)+R_{3,i}\\
&=R_{1,i}+R_{2,i}+R_{3,i}\\
&+\sum_{t+s+m+n\leq k,~ s,n \leq m}\mathbb{E}(Y_{i,t}Y_{i,t+s})\mathbb{E}(Y_{i,t+s+m}Y_{i,t+s+m+n})a(s,m,n)
\end{align*}
where by Proposition \ref{eq:eqmomentbounded} and the integral criteria
\begin{align*}
|R_{1,i}+R_{2,i}+R_{3,i}|&\leq 3 \cdot 24 C_{16}\sum_{t+s+m+n\leq k,~m,n\leq s} (1+s)^{-b}\\
&\leq 72 kC_{16}\sum_{s=0}^\infty (s+1)^2 c(1+s)^{-b} \leq \frac{ k C_{17}}{b-3}
\end{align*}
and therefore $\sup_{i=1,\ldots,N} \frac{(T-k)^2(T-j)^2}{T^6v_i^4}|R_{1,i}+R_{2,i}+R_{3,i}|\rightarrow 0$.
Now we turn towards the non vanishing part of $B_{1,i}.$ We change summation another time to arrive at
\begin{align}\label{versum}
&\sum_{t+s+m+n\leq k,~ s,n \leq m}\mathbb{E}(Y_{i,t}Y_{i,t+s})\mathbb{E}(Y_{i,t+s+m}Y_{i,t+s+m+n})a(l,m,n)\nonumber\\
&=\sum_{s=0}^k\sum_{n=0}^k3\sum_{m=\max(s,n)}^{k-s-n} (k-s-n-\max(s,n)-m)\gamma_i(s)\gamma_i(n)b(s,n)\nonumber\\
&=3\sum_{s=0}^k\sum_{n=0}^k\gamma_i(s)\gamma_i(n)b(s,n)\left((k-s-n-\max(s,n)+1)(k-s-n-\max(s,n))\right.\nonumber\\
&\left.-\frac{(k-s-n+1)(k-s-n)}{2}+\frac{\max(s,n)(\max(s,n)-1)}{2}\right)=\frac{3}{2}k^2v_i^4+3U_i
\end{align}
where $b(i,j)=\begin{cases}1&i,j=0\\
2~~ \mbox{for}&(i,j)=(0,1)\wedge (1,0)\\
4&i,j>0\end{cases}$\\ and the factor 3 appears since there are three possibilities to partition four random variables into two pairs.
The error $U_i$ consists of a sum of autocovariances which do not appear in \ref{versum} and one of autocovariances of improper quantity
\begin{align*}
|U_i|&\leq 12 k^2 \Big \vert \sum_{m=0}^\infty\sum_{n=k+1}^\infty \gamma_i(m)\gamma_i(n) \Big \vert +12\Big \vert \sum_{s=0}^k\sum_{n=0}^k \gamma(s)\gamma(n)\left[5k(s+n)+(s+n)^2\right] \big\vert\\
&\leq C_{18} k^2 \int_{k}^\infty (1+s)^{-b} ds+C_{19} k\left(\int_0^\infty (1+s)^{-b+1} ds\right)^2 +C_{20} \left(\int_0^\infty (1+s)^{-b+2} ds\right)^2 \\
&=C_{18}\frac{k^{-b+3}}{b-1}+C_{19}k\left(\frac{1}{b-2}\right)^2+C_{20}\left(\frac{1}{b-3}\right)^2
\end{align*}
where the squared integrals arise through the integral criteria applied to the separated sums. Therefore we get
$\sup_{i=1,\ldots,N} \frac{(T-k)^2(T-j)^2}{T^6v_i^4}|U_i|\rightarrow 0$.
By analogous calculations one gets
\begin{align*}
A_{1,1}&=\frac{(T-k)^2(T-j)^2(3k^2+k(j-k)}{2T^6}+R_{i,4}\\
A_{i,2}&=R_{i,5}\\
A_{i,3}&=\frac{(T-k)^2j^2k(T-j)}{2T^6}+R_{i,6}\\
A_{i,4}&=-2\frac{(T-k)k(T-j)^2 2k(j-k)}{2T^6}+R_{i,7}\\
A_{i,5}&=4\frac{(T-k)k(T-j)^2jk}{2T^6}+R_{i,8}\\
A_{i,6}&=R_{i,9}\\
A_{i,7}&=\frac{k^2(T-j)^2((T-j)(j-k)+k(j-k)+(T-j)k+3(j-k)^2)}{2T^6}+R_{i,10}\\
A_{i,8}&=-2\frac{k^2(T-j)^2j2(j-k)}{2T^6}+R_{i,11}\\
A_{i,9}&=\frac{k^2j^2(3(T-j)^2+(T-j)(j-k))}{2T^6}+R_{i,12}
\end{align*}
with $\sup_{i\in 1,\ldots,N} \sum_{k=4}^{12}|R_{i,k}|\rightarrow 0.$
Term manipulations yield the desired covariance structure.
\end{proof}
\begin{proof}[Step 3 of the Proof of Theorem \ref{T1}]
We show convergence of finite dimensional distributions. Let $k\in \mathbb{N}$ and $x_1,\ldots,x_k$ be arbitrary, by the Cramer Wold device it is sufficient (and necessary) to show convergence of linear combinations for arbitrary $\lambda_1,\ldots,\lambda_k$. By change of summation we get 
\begin{align*}
\sum_{j=1}^k \lambda_j W_{N,T}(x_j)=\frac{1}{N}\sum_{i=1}^N \underbrace{\sum_{j=1}^k\lambda_j \left(\left(S_T^{(i)}(x_j)\right)^2-\frac{\lfloor x_j T \rfloor(T-\lfloor x_j T \rfloor)}{T^2}\right)}_{D_i}
\end{align*}
where $D_i$ are independent but not identically distributed. Since $D_i$ also depends on $T,$ we apply a central limit theorem for random arrays of Lyapunov type (see for example \cite{serfling} p. 30). Therefore one has to show
\begin{align}\label{lapunov}
\frac{\sum_{i=1}^N\mathbb{E}\left(\sum_{j=1}^k \lambda_j[ S_i^2(x_j)-\mathbb{E}\{S_i^2(x_j)\}]\right)^4}
{\left(\sum_{i=1}^N\mathbb{E}\left[\sum_{j=1}^k \lambda_j\left\{ S_i^2(x_j)-\mathbb{E}(S_i^2(x_j))\right\}\right]^2\right)^2}\rightarrow 0.
\end{align}
For the nominator we repeatedly apply the $c_r$ inequality and Proposition \ref{eq:Bernsteinbound} to get
\begin{align*}
\sum_{i=1}^N&\mathbb{E}\left(\sum_{j=1}^k \lambda_j[ S_i^2(x_j)-\mathbb{E}\{S_i^2(x_j)\}]\right)^4 \\
&\leq
8^k\sum_{j=1}^k \lambda_j ^4\sum_{i=1}^N  \mathbb{E} ( S_i^2(x_j)-\mathbb{E}[S_i^2\{x_j\}])^4 \\
&\leq 8^{k+1}\sum_{j=1}^k \lambda_j ^4\sum_{i=1}^N  \mathbb{E} ( S_i^8(x_j))+\mathbb{E}(S_i^2(x_j))^4)\\
&\leq 8^{k+2}\sum_{j=1}^k \lambda_j ^4\sum_{i=1}^N  \left(\frac{1}{T^4}\mathbb{E} \left[\sum_{t=1}^{\lfloor x_jT \rfloor} Y_{i,t}\right]^8+\frac{1}{T^4}\mathbb{E} \left[\sum_{t=\lfloor x_jT \rfloor+1}^{T} Y_{i,t}\right]^8 \right.\\
&\left.+\frac{1}{T^4}\left[\mathbb{E}\left(\sum_{t=1}^{\lfloor x_j T \rfloor} Y_{i,t}\right)^2\right]^4+\frac{1}{T^4}\left[\mathbb{E}\left(\sum_{t=\lfloor x_j T \rfloor+1}^{T} Y_{i,t}\right)^2\right]^4\right)\\
&\leq 8^{k+2}\sum_{j=1}^k \lambda_j ^4 N (G_2 + \tilde{G}_2^4).
\end{align*}
For the denominator we exploit the cross-sectional independence and arrives at
\begin{align*}
\left(\sum_{i=1}^N\mathbb{E}\left(\sum_{j=1}^k \lambda_j( S_i^2(x_j)-\mathbb{E}(S_i^2(x_j)))\right)^2\right)^2&=\left(\sum_{i=1}^N \sum_{j,l=1}^k\lambda_j\lambda_l\mbox{Cov}(S_i^2(x_j),S_i^2(x_l))\right)^2\\
&=\left(\sum_{i=1}^N \underbrace{\sum_{j,l=1}^k \lambda_j,\lambda_l\mbox{Cov}(\Gamma(x_j),\Gamma(x_l))}_M +R_i\right)^2\\
&=M^2N^2+2MN \sum_{i=1}^NR_i+( \sum_{i=1}^NR_i)^2
\end{align*}
where $R_i$ denotes the remainder fulfilling $\sup_{i=1,\ldots,N}R_i\rightarrow 0$, see the second step of the proof of Theorem 1. Since the Gaussian process $\Gamma$ possesses a positive definite covariance function, we have $M>0$ and the denominator grows of the order $N^2$ while the nominator only grows linearly in $N$, which proofs (\ref{lapunov}) and hence the asymptotic normality. Together with step 1 and step 2 this proves that the finite dimensional distributions of $(W_{N,T}(x))_{x \in [0,1]}$ converge against that of $(\Gamma(x))_{x\in [0,1]}.$
\end{proof}
\begin{proof}[Step 4 of the Proof of Theorem \ref{T1}]
We want to apply the moment criteria of \cite{Billingsley} (see page 95) and therefore look at the difference between $W_{N,T}(x)$ and $W_{N,T}(y)$. By an expansion we see that we need bounds for four different moments
\begin{align*}
\mathbb{E}([W_{N,T}(x)-W_{N,T}(y)]^4)&=\frac{1}{N^2}\mathbb{E}\left(\left[\sum_{i=1}^N \underbrace{S_i^2(x)-\frac{\lfloor x T \rfloor(T-\lfloor x T \rfloor)}{T^2}-S_i^2(x)+\frac{\lfloor x T \rfloor(T-\lfloor x T \rfloor)}{T^2}}_{M_i}\right]^4\right)\\
&=\frac{1}{N^2}\sum_{i\neq j\neq k\neq l}\mathbb{E}(M_i)\mathbb{E}(M_j)\mathbb{E}(M_k)\mathbb{E}(M_l)
+\frac{1}{N^2}\sum_{i\neq j\neq k}\mathbb{E}(M_i^2)\mathbb{E}(M_j)\mathbb{E}(M_k)\\
&+\frac{1}{N^2}\sum_{i\neq j}\mathbb{E}(M_i^3)\mathbb{E}(M_j)+\frac{1}{N^2}\sum_{i\neq j}\mathbb{E}(M_i^2)\mathbb{E}(M_j^2)+\frac{1}{N^2}\sum_{i=1}^N\mathbb{E}(M_i^4).
\end{align*}

In the first step of the proof of Theorem \ref{T1} it it shown that 
\begin{align*}
|\mathbb{E}(M_i)|\leq \bigg|\mathbb{E}\left(S_i^2(x)-\frac{\lfloor x T \rfloor(T-\lfloor x T \rfloor)}{T^2}\right)\bigg|+ \bigg| \mathbb{E}\left(S_i^2(y)+\frac{\lfloor y T \rfloor(T-\lfloor y T \rfloor)}{T^2}\right)\bigg|\leq \frac{C_{21}}{T},
\end{align*}
where $C_{21}$ is independent of $i,~x$ and $y$. Let $\lfloor Tx\rfloor=u < \lfloor Ty\rfloor=v,$ for $\mathbb{E}(M_i^2)$ we apply Cauchy-Schwarz and $c_r$ inequality to get
\begin{align*}
\mathbb{E}(M_i^2)&\leq 2\mathbb{E}([S_i(x)^2-S_i(y)^2]^2)+2\left(\frac{\lfloor x T \rfloor(T-\lfloor x T \rfloor)}{T^2}-\frac{\lfloor y T \rfloor(T-\lfloor y T \rfloor)}{T^2}\right)^2\\
&\leq  \sqrt{\mathbb{E}([S_i(x)-S_i(y)]^4)}\sqrt{8\mathbb{E}([S_i(x)]^4)+8\mathbb{E}([S_i(y)]^4)}+2\left(\frac{(v-u)(u-T-v)}{T^2}\right)^2.
\end{align*}
Using Proposition \ref{eq:Bernsteinbound} one has
\begin{align*}
|\mathbb{E}(S_i(x)^4)|\leq\frac{8}{T^2\delta^4}\mathbb{E}\left(\frac{T-u}{T}\sum_{t=1}^u Y_{i,t}\right)^4+\frac{8}{T^2\delta^4}\mathbb{E}\left(\frac{u}{T}\sum_{t={u+1}}^T Y_{i,t}\right)^4\leq C_{22}
\end{align*}
and also
\begin{align*} 
\mathbb{E}([S_i(x)-S_i(y)]^4)&=\frac{8}{T^2\delta^4}\mathbb{E}\left(\left[\sum_{t=u+1}^v Y_{i,t}\right]^4\delta^4\right)+\frac{8(v-u)^4}{T^6}\mathbb{E}\left(\sum_{t=1}^T Y_{i,t}\right)^4\leq C_{23}\frac{(u-v)^2}{T^2}.
\end{align*}
Together we have $\mathbb{E}(M_i^2)\leq C_{24} \frac{v-u}{T}$ and analogously $\mathbb{E}(M_i^4)\leq C_{25} \frac{(v-u)^2}{T^2}$ respectively $|\mathbb{E}(M_i^3)|\leq C_{26}\frac{(v-u)^{3/2}}{T^{3/2}}.$ Since $N/T\rightarrow 0$, there is $C_{27}$ such that $N/T\leq C_{27}$\footnote{To emphasise that $N,T$ jointly tend to infinity it is maybe more convenient to use $N(T)$ instead of $N$ (respectively $T(N)$ instead of $T$). We have forgone on it for a better readability. However, the more elaborate notation is here superior. Since what is meant and used is that $N(T)/T\leq C_{27}, ~\forall T\in  \mathbb{N}$.} and we arrive at
\begin{align*}
\mathbb{E}([W_{N,T}(x)-W_{N,T}(y)]^4)&\leq \frac{1}{N^2}N^4\left(\frac{C_{21}}{T}\right)^4+\frac{1}{N^2}N^3C_{24}\frac{v-u}{T}\left(\frac{C_{21}}{T}\right)^2\\
&+\frac{1}{N^2}N^2 C_{26}\frac{(v-u)^{3/2}}{T^{3/2}}\frac{C_{21}}{T}
+\left(C_{24}\frac{v-u}{T}\right)^2+\frac{C_{25}}{N}\frac{(v-u)^2}{T^2}\\
&\leq C_{28} \frac{(v-u)^2}{T^2}+\frac{C_{29}}{T^2}+\frac{C_{30}(v-u)}{T^2}+\frac{C_{31}(v-u)^{3/2}}{T^{5/2}}\leq C_{32} |x-y|^2
\end{align*}
which proves tightness of the process.
\end{proof}

\begin{proof}[Proof of Theorem \ref{T2}]
We assume that $|Y_{i,1}|\leq c_1,~i=1,\ldots,N,$ though the proof is completely analogous in the unbounded case. We first calculate the mean squared error of $\hat{v}_i$. Equations for variance and bias of $\hat{v}_i$ are already known, see for example \cite{anderson} chapters eight and nine, though these are for known location. Denote therefore
\begin{align*}
\tilde{v}_i=\tilde{\gamma}_{i,0}+2\sum_{h=1}^{b_{T}}\tilde{\gamma}_{i,h}k\left(\frac{h}{b_{T}}\right)\mbox{~~with~~}
\tilde{\gamma}_i(h)=\frac{1}{T}\sum_{t=1}^{N-h}\left(Y_{i,t}-\mathbb{E}(Y_{i,1})\right)\left(Y_{i,t+h}-\mathbb{E}(Y_{i,1})\right)
\end{align*}
the long run variance with known location, where we continue to assume that $\mathbb{E}(Y_{i,1})=0,~i=1,\ldots,N.$
We denote $\overline{Y}_i=\frac{1}{T}\sum_{t=1}^T Y_{i,t}$ and describe $\hat{v}_i$ by $\tilde{v}_i$:
\begin{align*}
\hat{v_i}^2
&=\tilde{v}_i^2-\frac{1}{T}\sum_{t=1}^TY_{i,t}\overline{Y}_i-2\sum_{h=1}^{b_T} \frac{1}{T}\sum_{t=1}^{T-h}Y_{i,t+h}\overline{Y}_{i} k\left(\frac{h}{b_T}\right)\\
&-\frac{1}{T}\sum_{t=1}^TY_{i,t}\overline{Y}_i-2\sum_{h=1}^{b_T} \frac{1}{T}\sum_{t=1}^{T-h}Y_{i,t}\overline{Y}_{i} k\left(\frac{h}{b_T}\right)+\sum_{h=-b_T}^{b_T}\frac{1}{T}\overline{Y_i}^2k\left(\frac{h}{b_T}\right)\\
&=\tilde{v_i}^2+R_1+R_2+R_3+R_4+R_5.
\end{align*} 
Following the calculations of step 2 of the proof of Theorem \ref{T1} one finds upper bounds for the errors $R_1+R_2$ respectively $R_3+R_4$
\begin{align*}
E([R_1+R_2]^2)&\leq \frac{1}{T^4}6\sum_{h=0}^{b_T}\sum_{k=-b_T}^{b_T}(T^2v_i^2+C_{33}T) k\left(\frac{h}{b_T}\right)k\left(\frac{k}{b_T}\right)\leq C_{34}\frac{b_T^2}{T^2}
\end{align*}
and Proposition \ref{eq:Bernsteinbound} yields
\begin{align*}
E(R_3^2)\leq \frac{1}{T^6}\sum_{h=-b_T}^{b_T}\sum_{k=-b_T}^{b_T}T^2 C_{35} k\left(\frac{h}{b_T}\right)k\left(\frac{k}{b_T}\right)\leq \frac{b_T^2}{T^4}C_{36}.
\end{align*}
Therefore we get
\begin{align}\label{vermittel}
\mathbb{E}\left([\hat{v}_i^2-\tilde{v}_i^2]^2\right)\leq C_{37} \frac{b_T^2}{T^2}.
\end{align}
In the next step we calculate the bias of $\tilde{v}_i^2$. straightforward calculations yield:
\begin{align*}
|E(\tilde{v}_i^2)-v_i^2|&=\bigg|E\left(\frac{1}{T}\sum_{t=1}^TY_{i,t}^2+2\frac{1}{T}\sum_{h=1}^{b_T}\sum_{t=1}^{T-h}Y_{i,t}Y_{i,t+h}k\left(\frac{h}{b_T}\right)\right)-\sum_{h=-\infty}^\infty \gamma_i(h) \bigg|\\
&\leq \bigg|2\sum_{h=1}^{b_T}\left( \left(1-\frac{h}{T}\right)k\left(\frac{h}{b_T}\right)-1\right) \gamma_i(h)\bigg|+2\bigg|\sum_{h=b_T+1}^\infty \gamma_i(h)\bigg|\\
&\leq 2\bigg|\sum_{h=1}^{b_T}\left(1- k\left(\frac{h}{b_T}\right)\right)\gamma_i(h)\bigg|+2\bigg|\sum_{h=1}^{b_T} \gamma_i(h) k\left(\frac{h}{b_T}\right)\frac{h}{T}\bigg|+2\bigg|\sum_{h=b_T+1}^\infty  \gamma_i(h)\bigg|\\
&=A_1+A_2+A_3.
\end{align*}
The sum $A_1$ describes the error which is generated by the kernel. To bound the error, we develop the kernel $k$ around 0. Since the first $m-1$ derivatives are 0 by Assumption 2 iv), the first non-vanishing term is of order $m.$ However to ensure that the remainder is negligible, we only develop the Taylor series up to order $s-1:$ 
\begin{align}
|A_1|&\leq \sum_{h=1}^{b_T}C_{38}(1+h)^{-b} \left( \frac{h}{b_T}\right)^s \leq \frac{1}{b_T^s}C_{39} \sum_{h=1}^\infty (1+h)^{-b+s}=\frac{1}{b_T^s}C_{40}.
\end{align}
Assumption 1 2) a) ii) and the integral criterion yields
$|A_2|\leq \frac{C_{41}}{T}~\mbox{and}~|A_3|\leq C_{42} b_T^{-b+1}$ and therefore 
\begin{align}\label{biaslong}
|E(\tilde{v}_i^2)-v_i^2|\leq \frac{1}{b_T^s}C_{40}+\frac{C_{41}}{T}+C_{42} b_T^{-b+1}.
\end{align}

Now we turn our attention to the variance of $\tilde{v}_i$, where the following expansion is known (see for example \cite{anderson}, page 528, chapter 9.3.3)
\begin{align*}
Var(\tilde{v}_i)&=\frac{1}{T}\sum_{g,h=-b_T}^{b_T}k\left(\frac{h}{b_T}\right)k\left(\frac{g}{b_T}\right)\sum_{r=-T+1}^{T-1}\phi(r,g,h)\\
&\cdot \left(\gamma_i(r)\gamma_i(r+h-g)+\gamma_i(r-g)\gamma_i(r+g)+\kappa_i(h,-r,g-r)\right)\\
&=B_1+B_2+B_3
\end{align*} 
where $\kappa_i(r,s,t)=\mathbb{E}(Y_{i,1}Y_{i,r+1}Y_{i,s+1},Y_{i,t+1})-\gamma(r)\gamma(t-r)-\gamma(s)\gamma(t-r)-\gamma(t)\gamma(s-r)$ for $r,s,t\in \mathbb{N}$ is the fourth order cumulant and the formal definition of $\phi(r,g,h)$ can be found on page 528 of \cite{anderson}. For us it is only important that $|\phi(r,g,h)|\leq 1.$
Following the calculations in \cite{anderson} we get
\begin{align*}
|B_2|&\leq \left| \frac{1}{T}\sum_{g,h=-b_T}^{b_T}\sum_{r=\max(g-b_T,h-b_T,-T+1)}^{\min(g+b_T,h+b_T,T-1)}\phi_T(r,r-g,h-r)k\left(\frac{r-g}{b_T}\right)k\left(\frac{h-r}{b_T}\right)\gamma_i(g)\gamma_i(h)\right|\\
&+\left( \frac{8b_T}{T}+\frac{4}{b_TT}\right)\sum_{g=-\infty}^\infty \sum_{h=b_T+1}^\infty R |\gamma_i(g)\gamma_i(h)|=D_1+D_2
\end{align*}
where Proposition \ref{eq:eqmomentbounded} yields $|D_1|\leq C_{43} \frac{b_T}{T}$, $|D_2|\leq C_{44}\frac{b_T}{T}$ and similar arguments reveal $|B_1|\leq C_{45}\frac{b_T}{T}.$ We rearrange the sum of cumulants  and split the expectation with Proposition \ref{eq:eqmomentbounded} where the lag difference is largest
\begin{align*}
|B_3|&\leq \frac{8R}{T} \sum_{r,s,t=0}^\infty |\kappa(r,s,t)|\\
&\leq \frac{24R}{T}\sum_{r,s,t=0}^\infty |\kappa(r,r+s,r+s+t)|\\
&\leq \frac{24R}{T} \sum_{r=0}^\infty \sum_{s,t\leq r} |\mathbb{E}(Y_{i,1})\mathbb{E}(Y_{i,1+r}Y_{i,1+r+s}Y_{i,1+r+s+t})-\gamma_i(r)\gamma_i(t)|\\
&+\frac{24R}{T} \sum_{s=0}^\infty \sum_{r,t\leq s} |\mathbb{E}(Y_{i,1}Y_{i,r+1})\mathbb{E}(Y_{i,1+r+s}Y_{i,1+r+s+t})-\gamma_i(r)\gamma_i(t)|\\
&+\frac{24R}{T} \sum_{t=0}^\infty \sum_{r,s\leq t} |\mathbb{E}(Y_{i,1}Y_{i,r+1}Y_{i,1+r+s})\mathbb{E}(Y_{i,1+r+s+t})-\gamma_i(r)\gamma_i(t)|\\
&+\frac{24R}{T}\sum_{r,s,t=0}^\infty |\gamma_i(r+s)\gamma_i(s+t)|+\frac{24R}{T}\sum_{r,s,t=0}^\infty |\gamma_i(r+s+t)\gamma_i(s)|=F_1+F_2+F_3+F_4+F_5.
&\end{align*}
We look exemplarily at $F_1,F_2$ and $F_4$: 
\begin{align*}
F_1&\leq\frac{C_{46}}{T}\sum_{r=0}^\infty r^2(1+r)^{-b}+\frac{C_{47}}{T}\sum_{r=0}^\infty r(1+r)^{-b}\leq \frac{C_{48}}{T}\\
F_2&\leq\frac{C_{49}}{T}\sum_{s=0}^\infty r^2(1+s)^{-b}+\frac{C_{50}}{T}\sum_{s=0}^\infty\sum_{r=s+1}^\infty (1+r)^{-b} \sum_{t=0}^\infty (1+t)^{-b}\\
&\leq \frac{C_{51}}{T}+ \sum_{s=0}^\infty \frac{C_{52}}{T}(1+s)^{-b+1}\leq \frac{C_{53}}{T}\\
F_4&= \frac{C_{54}}{T}\sum_{s=0}^\infty\sum_{r=s+1}^\infty (1+r)^{-b+1} \sum_{t=s+1}^\infty (1+t)^{-b+1}\leq \frac{C_{55}}{T}
\end{align*}
So we finally arrive at
\begin{align}\label{varestmi}
\mbox{Var}(\tilde{v}_i)\leq \frac{C_{56}b_T}{T}
\end{align}
and therefore by (\ref{vermittel}), (\ref{biaslong}) and (\ref{varestmi}) 
\begin{align}\label{eq:mselrv}
\mathbb{E}([\hat{v}_i^2-v_i^2]^2)\leq C_{57}\frac{b_T^2}{T^2}+C_{58}b_T^{-2s}+C_{59}b_T^{-2b+2}+C_{60}\frac{b_T}{T}.
\end{align}
Now we show that the long run variance estimations $\hat{v}_i$ are bounded from below for large $T,N$. Denote by $D_{N,T}$ the event that $\hat{v}_i^2>v_i^2/2,~i=1,\ldots,N,$ then
\begin{align*}
P(D_{N,T})&\geq 1-\sum_{i=1}^N P(\hat{v}_i^2<v_i^2/2)\\
&\geq 1-\sum_{i=1}^N P(|\hat{v}_i^2-v_i^2|\geq v_i^2/2)\\
&\geq 1-4\left(C_{57}\frac{b_T^2}{T^2}+C_{58}b_T^{-2m}+C_{59}b_T^{-2b+2}+C_{60}\frac{b_T}{T}\right)\sum_{i=1}^N\frac{1}{v_i^4}\rightarrow 1
\end{align*}
and therefore it is enough to prove Theorem 2 on $D_{N,T}$. In the following we show that the difference between $W_{N,T}(x)$ and $\tilde{W}_{N,T}(x)$ is negligible for every $x \in [0,1].$ By an expansion one can extract the error of the long run variance estimation from the difference
\begin{align}\label{diffvar}
W_{N,T}&(x)-\tilde{W}_{N.T}(x)\nonumber\\
&=\frac{1}{\sqrt{N}}\sum_{i=1}^N \left(\frac{1}{\hat{v}_i^2}-\frac{1}{v_i^2} \right)\left(\frac{1}{T}\left[\sum_{t=1}^{\lfloor Tx \rfloor} Y_{i,t}
-\frac{\lfloor  Tx \rfloor}{T} \sum_{t=1}^T Y_{i,t}\right]^2 -v_i^2\frac{\lfloor Tx\rfloor (T-\lfloor Tx\rfloor )}{T^2}\right)\nonumber\\
&+\frac{\lfloor Tx\rfloor (T-\lfloor Tx\rfloor )}{T^2}\frac{1}{\sqrt{N}}\sum_{i=1}^N\frac{v_i^2-\hat{v}_i^2}{\hat{v_i}^2}=G_1+G_2.
\end{align}
The second term of (\ref{diffvar}) can be bounded by (\ref{vermittel}), (\ref{biaslong}) and (\ref{eq:mselrv})
\begin{align*}
\mathbb{E}\left(G_2^2\right)&\leq \frac{1}{N}\sum_{i\neq j} \frac{4}{v_i^2v_j^2} |\mathbb{E}(\hat{v}_i^2-v_i^2)\mathbb{E}(\hat{v}_j^2-v_j^2)|+\frac{1}{N}\sum_{i=1}^N \frac{4}{v_i^4}\mathbb{E}(\hat{v}_i^2-v_i^2)^2\\
&\leq C_{61}N\left(\frac{b_T^2}{T^2}+\frac{1}{b_T^{2s}}+\frac{1}{T^2}+b_T^{-2b+2}\right)+C_{62}\left( \frac{b_T^2}{T^2}+b_T^{-2s}+b_T^{-2b+2}+\frac{b_T}{T}\right)\rightarrow 0.
\end{align*}
Using the Cauchy-Schwarz inequality and (\ref{eq:mselrv}) one gets for $E_1:$
\begin{align*}
\mathbb{E}(|E_1|)&\leq \frac{1}{\sqrt{N}}\sum_{i=1}^N \frac{4}{\delta^4}\sqrt{\mathbb{E}([\hat{v_i}^2-v_i^2]^2)}\\
&\sqrt{\mathbb{E}\left(\left[\frac{1}{T}\left\{\sum_{t=1}^{\lfloor Tx \rfloor} Y_{i,t}
-\frac{\lfloor  Tx \rfloor}{T} \sum_{t=1}^T Y_{i,t}\right\}^2 -v_i^2\frac{\lfloor Tx\rfloor (T-\lfloor Tx\rfloor )}{T^2}\right]^2\right)}\\
&\leq \sqrt{N}\sqrt{C_{62}\left( \frac{b_T^2}{T^2}+b_T^{-2s}+b_T^{-2b+2}+\frac{b_T}{T}\right)}\sqrt{\left(1+\frac{C_{63}}{T^2}\right)}\rightarrow 0
\end{align*}
where $C_{63}$ can be calculated based on step 1 and 2 of the Proof of Theorem 1. The tightness of $(\tilde{W}_{N,T}(x)-W_{N,T}(x))_{x\in [0,1]}$ can be proved like the tightness of $W_{N,T}(x)$ using that $\hat{v_i}$ is bounded from below, which completes the proof.
\end{proof}
The proof of Theorem \ref{storung} consists of three steps:
\begin{itemize}
\item showing pointwise convergence $\check W_{N,T}(x)-\check W_{N,T}(x)\rightarrow 0,~\forall x\in [0,1],$ under known long run variances,
\item proving tightness of $(\check{W}_{N,T}(x)-\check{W}_{N,T}(x))_{x\in [0,1]}$ under known long run variances,
\item verifying $\sup_{x \in [0,1]}|\check{W}_{N,T}(x)-\check{W}_{N,T}(x)|\rightarrow 0$ under estimated long run variances.
\end{itemize}
First denote $E_{N,T}$ the event that $\hat{\sigma}_i\geq \sigma_i/2$ for $i=1,\ldots,N,$ then
\begin{align}\label{ungl1}
P(E_{N,T})\geq 1-\sum_{i=1}^N P(|\hat{\sigma}_i-\sigma_i|>\sigma_i/2)\geq 1-c_2NT^{-\alpha}.
\end{align}
We use \ref{ungl1} to bound the probability $F_{N,T}$, which is the event that $|\hat{\mu}_i-\mu_i|\leq c_1T^{-\beta}$ and $|\frac{1}{\hat{\sigma}_i}-\frac{1}{\sigma_i}|\leq 4\frac{c_1}{\delta^2} T^{-\beta}$ for $i=1,\ldots,N:$
\begin{align*}
P(F_{N,T})&\geq 1-\sum_{i=1}^N c_2T^{-\alpha}-\sum_{i=1}^N \left[P\left(\frac{\hat{\sigma}_i-\sigma_i}{\sigma_i\hat{\sigma}_i}|E_{N,T}\right)P(E_{N,T})-P\left(\frac{\hat{\sigma}_i-\sigma_i}{\sigma_i\hat{\sigma}_i}|\overline{E}_{N,T}\right)P(\overline{E}_{N,T})\right]\\
&\geq 1-c_2NT^{-\alpha}-c_2NT{-\alpha}-c_2NT^{-\alpha}\rightarrow 0
\end{align*}
and so we can proof Theorem \ref{storung} under $F_{N,T}$. Let furthermore w.l.o.g $\mu_i=0,~\sigma_i=1,~i=1,\ldots,N.$ 
\begin{proof}[Proof of step 1 of Theorem \ref{storung}]
The expansion
\begin{align*}
\Psi_i\left( \frac{X_{i,t}-\hat{\mu}_i}{\hat{\sigma}_i}\right)=\Psi_i\left(X_{i,t}+X_{i,t}\left(\frac{1}{\hat{\sigma}}-1\right)-\hat{\mu}_i\left(\frac{1}{\hat{\sigma}_i}-1\right)-\hat{\mu}_i\right)
\end{align*}
is almost impossible to work with, since $\hat{\mu}_i$ and $\hat{\sigma}_i$ depend on $(X_{i,t})_{t=1,\ldots,T}$. Instead one can look at
$
\Psi_i\left(X_{i,t}+X_{i,t}dT^{-\beta}+eT^{-\beta}\right)=Z_{i,T}(d,e),
$
where $d$ and $e$ are non random with $|d|\leq c_1$ respectively $e\leq \max(c_1,4\frac{c_1}{\delta^2})$, since we are in the case of $D_{N,T}$. Denote
\begin{align*}
\bar{S}^{(i)}_{T,d,e}(x)=\frac{1}{\sqrt{T}v_i}\left(\sum_{t=1}^{\lfloor Tx \rfloor}Z_{i,t}(d,e)-\frac{\lfloor Tx \rfloor}{T}\sum_{t=1}^T Z_{i,t}(d,e)\right),~~x\in[0,1]
\end{align*}
the disturbed CUSUM-statistic
\begin{align*}
\bar{W}_{N,T,d,e}(x)=\frac{1}{\sqrt{N}}\sum_{i=1}^N \left(\left(\bar{S}_{T,d,e}^{(i)}(x)\right)^2-\frac{\lfloor x T \rfloor(T-\lfloor x T \rfloor)}{T^2}\right),~~x\in [0,1]
\end{align*}
the related panel-CUSUM. To prove 
\begin{align*}
\sup_{d,e\in[-c_1,c_1]}|W_{N,T}(x)-\bar{W}_{N,T,d,e}(x)|\rightarrow 0,~~ x\in (0,1)\end{align*}
we have to bound the difference between the squared individual CUSUM-statistics
\begin{align*}
\left(\bar{S}^{(i)}_{T,d,e}\right.&\left.(x)\right)^2-\left(S^{(i)}(x)\right)^2\\
&=\frac{1}{\sqrt{T}v_i}\left(\sum_{t=1}^k Z_{i,T}(d,e)-\frac{k}{T}\sum_{t=1}^T Z_{i,T}(d,e)-\sum_{t=1}^k \Psi_i(X_{i,t})+\frac{k}{T}\sum_{t=1}^T \Psi_i(X_{i,t})\right)\\
&\cdot\frac{1}{\sqrt{T}v_i}\left(\sum_{t=1}^k Z_{i,T}(d,e)-\frac{k}{T}\sum_{t=1}^T Z_{i,T}(d,e)+\sum_{t=1}^k \Psi_i(X_{i,t})-\frac{k}{T}\sum_{t=1}^T \Psi_i(X_{i,t})\right)=A_iB_i.
\end{align*}
We will show that $\mathbb{E}\left(B_i^2\right)$ is bounded while $\mathbb{E}\left(A_i^2\right)$ converges to 0 sufficiently fast. Essential in both calculations is a Taylor expansions of order $m$ for $Z_{i,t}(d,e)$ around $X_{i,t}$ for $t=1,\ldots,T:$
\begin{align}\label{Taylor}
Z_{i,t}(d,e)&=\Psi_i(X_{i,t})+\sum_{k=1}^m\frac{ \Psi_i(X_{i,t})^{(m)}}{m!}\left(X_{i,t}dT^{-\beta}+eT^{-\beta}\right)^m\nonumber\\
&+\frac{\Psi_i(\xi_{i,t})^{(m+1)}}{(m+1)!}\left(X_{i,t}dT^{-\beta}+eT^{-\beta}\right)^{m+1}
\end{align} 
for some $\xi_{i,t}\in [X_{i,t}-X_{i,t}dT^{-\beta}-eT^{-\beta},X_{i,t}+X_{i,t}dT^{-\beta}+eT^{-\beta}].$ Denote 
\begin{align*}U_i^{(r,s)}(k)=\begin{cases}\sum_{t=1}^k \Psi_i^{(r)}(X_{i,t})X_{i,t}^s-\frac{k}{T}\sum_{t=1}^T \Psi_i^{(r)}(X_{i,t})X_{i,t}^s&r=0,\ldots,m\\
\sum_{t=1}^k \Psi_i^{(r)}(\xi_{i,t})X_{i,t}^s-\frac{k}{T}\sum_{t=1}^T \Psi_i^{(r)}(\xi_{i,t})X_{i,t}^s&r=m+1
\end{cases},s\leq m
\end{align*}
the non standardized CUSUM statistics which arises from the $r-th$ Taylor-summand in (\ref{Taylor}). In $A_i$ the original CUSUM $U_i^{(0,0)}(k)$ cancels out and by repeated application of the $c_r$ inequality we obtain:
\begin{align*}
\mathbb{E}(A_i^2)&\leq \sum_{r=1,\ldots,m+1,s\leq m}\frac{2^{(m+1)(m+2)/2-1}}{Tv_i^2} d^{2s}e^{2(r-s)}T^{-2\beta r}\frac{{r \choose s}}{r!}\mathbb{E}\left(\left[U_i^{(r,s)}(k)\right]^2\right)
\end{align*}
Because of assumption i) of Theorem \ref{storung} the processes $\left(\Psi_i(X_{i,t})^{(r)}X_{i,t}^s\right)_{t\in \mathbb{N}}$ is also P-NED for $r=1,\ldots,m$ and $s\leq m$, so we can apply the same calculations as in step 1 of the proof of Theorem \ref{T1} and receive $\frac{1}{T}\mathbb{E}\left(\left[U_i^{(r,s)}(k)\right]^2\right)\leq C_{63}+\frac{C_64}{T}$ for $r=1,\ldots,m$. For $m+1$ we use the upper bounds $c$ and $d$ of assumption i) and ii) of Theorem \ref{storung} to obtain $\frac{1}{T^2}\mathbb{E}\left(\left[U_i^{(r,s)}(k)\right]^2\right)\leq C_{63}.$ So we have $\mathbb{E}(A_i^2)\leq C_{65}T^{-2\beta}+2C_{64}T^{-2\beta(m+1)+1}$ and by the same calculations $\mathbb{E}(B_i^2)\leq C_{66}+C_{67}T^{-2\beta}+C_{68}T^{-2\beta(m+1)+1}$. Finally we apply the Cauchy-Schwarz inequality to obtain
\begin{align*}
\sup_{d,e\in[-c_1,c_1]}E(|W_{N,T}(x)-\bar{W}_{N,T,d}(x)|)&\leq \frac{1}{\sqrt{N}}\sum_{i=1}^N \left(C_{69}T^{-\beta}+C_{70}T^{-(m+1)\beta+1/2}\right)\rightarrow 0
\end{align*}
which implies pointwise convergence.
\end{proof} 
\begin{proof}[Step 2 of the proof of Theorem \ref{storung}]
Denote $Tx=u>v=Ty,$ like in step 4 of the proof of Theorem \ref{T1}, we want to apply the tightness critierion of \cite{Billingsley}. But before we use the Taylor expansion (\ref{Taylor})
\begin{align*}
W_{N,T}(x)-\bar{W}_{N,T,d}(x)&=\frac{1}{\sqrt{N}}\sum_{i=1}^N \left[\frac{1}{Tv_i^2}U_i^{(0,0)}(u)^2-\frac{1}{Tv_i^2}\left(\sum_{r=0,\ldots,m+1,s\leq r} a_{r,s}U_i^{(r,s)}(u)\right)^2\right]\\
&=\sum_{r=0,\ldots,m+1,s\leq r}\sum_{n=0,\ldots,m+1,p\leq n}\frac{1}{\sqrt{N}}\sum_{i=1}^N\frac{1}{Tv_i^2} a_{r,s}a_{n,p}U_i^{(r,s)}(u)U_i^{(n,p)}(u)
\end{align*}
where $a_{r,s}=\frac{d^se^{r-s}T^{-\beta r}{r \choose s}}{r!}$ to see that we only
need to proof tightness for the individual summands 
\begin{align*}\left(A_{N,T}(x)^{rsnp}\right)_{x\in [0,1]}&=\left(\frac{1}{\sqrt{N}}\sum_{i=1}^N \frac{1}{Tv_i^2}a_{r,s}a_{n,p}U_i^{(r,s)}(\lfloor Tx \rfloor )U_i^{(n,p)}(\lfloor Tx \rfloor )\right)_{x\in [0,1]}
\end{align*}
for $r,n=1,\ldots,m+1,$ $s\leq r$ and $p\leq n .$ Denote 
\begin{align*}M_i^{(rsnp)}&=\frac{1}{Tv_i^2}a_{r,s}a_{n,p}U_i^{(r,s)}(u)U_i^{(n,p)}(u)-\frac{1}{Tv_i^2}a_{r,s}a_{n,p}U_i^{(r,s)}(v)U_i^{(n,p)}(u),
\end{align*}
then we expand the difference of the fourth moment to
\begin{align*}
\mathbb{E}&(|A_{N,T}(x)^{rsnp}-A_{N,T}(y)^{rsnp}|^4)\\
&=\frac{1}{N^2}\sum_{t\neq u\neq v\neq w}\mathbb{E}(M_t^{(rsnp)})\mathbb{E}(M_u^{(rsnp)})\mathbb{E}(M_v^{(rsnp)})\mathbb{E}(M_w^{(rsnp)})\\
&+\frac{1}{N^2}\sum_{t\neq u\neq v}\mathbb{E}\left(\left[M_t^{(rsnp)}\right]^2\right)\mathbb{E}(M_u^{(rsnp)})\mathbb{E}(M_v^{(rsnp)})+\frac{1}{N^2}\sum_{t\neq u}\mathbb{E}\left(\left[M_t^{(rsnp)}\right]^3\right)\mathbb{E}(M_u^{(rsnp)})\\
&+\frac{1}{N^2}\sum_{t\neq u}\mathbb{E}\left(\left[M_t^{(rsnp)}\right]^2\right)\mathbb{E}\left(\left[M_u^{(rsnp)}\right]^2\right)+\frac{1}{N^2}\sum_{t=1}^N\mathbb{E}\left(\left[M_t^{(rsnp)}\right]^4\right).
\end{align*}
Exemplarily we look at $\mathbb{E}(M_i^{(jk)})$. If $j,k<m+1$ we can apply Proposition \ref{eq:Bernsteinbound} to get
\begin{align*}
|\mathbb{E}(M_i^{(rsnp)})|&= \Bigg|\frac{a_{r,s}a_{n,p}}{Tv_i^2}\mathbb{E}\left(\left[\sum_{t=1}^{\lfloor Tx \rfloor}\Psi_i^{(r)}(X_{i,t})X_{i,t}^s-\frac{\lfloor Tx \rfloor }{T}\sum_{t=1}^T \Psi_i^{(r)}(X_{i,t})X_{i,t}^s\right]\right.\\
&\cdot \left[\sum_{t=\lfloor Ty \rfloor+1}^{\lfloor Tx \rfloor }\Psi_i^{(n)}(X_{i,t})X_{i,t}^p-\frac{\lfloor Tx \rfloor -\lfloor Ty \rfloor}{T}\sum_{t=1}^T \Psi_i^{(n)}(X_{i,t})X_{i,t}^p\right]\\
&-\left[\sum_{t=1}^{\lfloor Ty \rfloor}\Psi_i^{(n)}(X_{i,t})X_{i,t}^p-\frac{\lfloor Ty \rfloor }{T}\sum_{t=1}^T \Psi_i^{(n)}(X_{i,t})X_{i,t}^p\right]\\
&\cdot \left. \left[\sum_{t=\lfloor Ty \rfloor +1}^{\lfloor Tx \rfloor }\Psi_i^{(r)}(X_{i,t})X_{i,t}^{s}-\frac{\lfloor Tx \rfloor-\lfloor Ty \rfloor }{T}\sum_{t=1}^T \Psi_i^{(r)}(X_{i,t})X_{i,t}^{s}\right]\right)\Bigg|\\
&\leq \frac{a_{r,s}a_{n,p}}{\delta}\left(\sqrt{C_{71}\lfloor Tx \rfloor}\sqrt{C_{72}(\lfloor Tx \rfloor-\lfloor Ty \rfloor})+\sqrt{C_{73}\lfloor Tx \rfloor}\sqrt{C_{74}(\lfloor Tx \rfloor-\lfloor Ty \rfloor})\right)\\
&\leq C_{78}T^{-\beta(r+n)}\sqrt{x-y}
\end{align*}
Analogously one can prove $\mathbb{E}([M_i^{(rsnp)}]^2)\leq T^{-2\beta(n+p)} C_{79}(x-y)$, $|\mathbb{E}([M_i^{rsnp}]^3)|\leq C_{80} T^{-3\beta(j+k)}(x-y)^\frac{3}{2}$ and $\mathbb{E}([M_i^{rsnp}]^4)\leq C_{81}T^{-4\beta(n+p)}(x-y)^2$ so that
\begin{align*}
\mathbb{E}&(|A_{N,T}(x)^{rsnp}-A_{N,T}(y)^{rsnp}|^4)\leq C_{82}T^{-4\beta(j+k)}(x-y)^2\left(N^2+N+1+\frac{1}{N}\right)\rightarrow 0
\end{align*}
which proves tightness. We cannot use Proposition \ref{eq:Bernsteinbound} for the Taylor-remainders which  occur if $j=m+1$ or $k=m+1$. In this cases we use assumption ii) of Theorem \ref{storung} to get
\begin{align}\label{trivial2}
\mathbb{E}\left\{\left[\sum_{i=1}^T \Psi_i^{(m+1)}(\xi_{i,t})X_{i,t}^s\right]^k\right\}\leq T^k d.
\end{align}
So if only one of $r,n$ equals $m+1$ we have
\begin{align*}
\mathbb{E}&(|B_{N,T}(x)^{rs(m+1)p}-B_{N,T}(y)^{rs(m+1)p)}|^4)\leq C_{83}T^{-4\beta(r+m+1)+2}(x-y)^2\left(N^2+N+1+\frac{1}{N}\right)\rightarrow 0
\end{align*}
respectively if both are $m+1$
\begin{align*}
\mathbb{E}(|B_{N,T}(x)^{(m+1)s(m+1)p}&-B_{N,T}(y)^{(m+1)s(m+1)p}|^4)\\
&\leq C_{84}T^{-8\beta(m+1)+4}(x-y)^2\left(N^2+N+1+\frac{1}{N}\right)\rightarrow 0.
\end{align*}
\end{proof}
\begin{proof}[Step 3 of the proof of Theorem \ref{storung}]
We continue to assume that w.l.o.g. $\mu_i=0,\sigma_i=1$ for $i=1,\ldots,N.$  The proof follows that of Theorem \ref{T2}. We first show that the difference between the long run variance with known standardization $\hat{v}_i$ and the long run variance with estimated standardization
\begin{align*}
\check{v}_i=\check{\gamma}_{i,0}+2\sum_{i=0}^{b_{i,T}}\check{\gamma}_{i,h}k\left(\frac{h}{b_{T}}\right),~i=1,\ldots,N,
\end{align*}
with
\begin{align*}
\check{\gamma}_i(h)=\frac{1}{T}\sum_{t=1}^{N-h}\left(\Psi_i\left[\frac{X_{i,t}-\mu_T^{(i)}}{\sigma_T^{(i)}}\right]-\check{Y}_i\right)\left(\Psi_i\left[\frac{X_{i,t}-\mu_T^{(i)}}{\sigma_T^{(i)}}\right]-\check{Y}_i\right),~h=0,\ldots,T-1, 
\end{align*}
and $\check{Y}_i=\frac{1}{T}\sum_{t=1}^T \Psi_i\left(\frac{X_{i,t}-\mu_T^{(i)}}{\sigma_T^{(i)}}\right)$
converge with rate $T^{-\min(2\beta,2\beta(m+1)-1)}.$ We use the Taylor series (\ref{Taylor}) and denote
$
Z_{i,t}^{(k,l)}=\Psi_i\left(X_{i,t}\right)^{(k)}X_{i,t}^l-\frac{1}{T}\sum_{r=1}^T \Psi_i\left(X_{i,r}\right)^{(k)}X_{i,t}^l
$ 
to obtain
\begin{align}\label{Zerlegung}
\hat{v}_i^2-\check{v}_i^2&=\frac{1}{T}\sum_{s.t=1}^T\left(\Psi_i(X_{i,t})-\frac{1}{T}\sum_{r=1}^T \Psi_i(X_{i,r}) \right)\left(\Psi_i(X_{i,s})-\frac{1}{T}\sum_{r=1}^T \Psi_i(X_{i,r}) \right)k\left(\frac{s-t}{b_T}\right)\nonumber\\
&-\sum_{k=0,l\leq k}^{m+1}\sum_{n=0,p\leq n}^{m+1}a_{k,l}a_{n,p}\frac{1}{T}\sum_{s.t=1}^T Z_{i,t}^{(k,l)}Z_{i,t}^{(n,p)}k\left(\frac{s-t}{b_T}\right).
\end{align}
So by (\ref{Zerlegung}) and the $c_r$ inequality we see
\begin{align*}
\mathbb{E}\left(\left[\hat{v}_i^2-\check{v}_i^2\right]^2\right)\leq 2^{\frac{(m+2)(m+1)}{2}-1}\sum_{k=0,l\leq k}^{m+1}\sum_{n=0,p\leq n}^{m+1}a_{k,l}^2a_{n,p}^2\frac{1}{T^2}\mathbb{E}\left(\left[\sum_{s,t=1}^T Z_{i,t}^{(k,l)}Z_{i,s}^{(n,p)}k\left(\frac{s-t}{b_T}\right)\right]^2\right)
\end{align*}
and
\begin{align*}
 \frac{1}{T^2}&\mathbb{E}\left(\left[\sum_{s,t=1}^T Z_{i,t}^{(k,l)}Z_{i,s}^{(n,p)}k\left(\frac{s-t}{b_T}\right)\right]^2\right)\\
 &\leq \frac{2^4}{T^2}\mathbb{E}\left(\left[\sum_{s,t=1}^T\Psi(X_{i,t})^{(k)}X_{i,t}^l\Psi(X_{i,s})^{(n)}X_{i,s}^pk\left(\frac{s-t}{b_T}\right)\right]^2\right)\\
&+\frac{2^4}{T^4}\mathbb{E}\left(\left[\sum_{s,t,r=1}^T\Psi(X_{i,t})^{(k)}X_{i,t}^l\Psi(X_{i,r})^{(n)}X_{i,r}^pk\left(\frac{s-t}{b_T}\right)\right]^2\right)\\
&+\frac{2^4}{T^4}\mathbb{E}\left(\left[\sum_{s,t,r=1}^T\Psi(X_{i,r})^{(k)}X_{i,r}^l\Psi(X_{i,s})^{(n)}X_{i,s}^pk\left(\frac{s-t}{b_T}\right)\right]^2\right)\\
&+\frac{2^4}{T^6}\mathbb{E}\left(\left[\sum_{s,t,r,u=1}^T\Psi(X_{i,r})^{(k)}X_{i,r}^l\Psi(X_{i,u})^{(n)}X_{i,u}^pk\left(\frac{s-t}{b_T}\right)\right]^2\right)=A_1+A_2+A_3+A_4.
\end{align*}
We exemplarily look at $A_1.$ Let $r,n<m+1$, like in the proof of Proposition \ref{eq:Bernstein} we want to rearrange the summands and split the expectations where the time lag is largest by applying Proposition \ref{eq:eqmomentboundeddifferent}. Here we have the problem that we apply two different functions $g_1(x)=\Psi^{(k)}(x)x^l$ and $g_2(x)=\Psi^{(n)}(x)x^p$ to the random variables $X_{i,t},$ so after the rearrangement we have six different sums, the one where $g_1$ is applied to the two random variables with two smallest indices and $g_2$ to the other and so on:
\begin{align*}
A_1&\leq 4\sum_{\begin{matrix}a,b,c,d \in \{1,2\}\\
a+b+c+d=6
\end{matrix}}\sum_{1\leq s\leq t \leq u \leq v\leq T}|\mathbb{E}[g_a(X_{i,s}),g_b(X_{i,t}),g_c(X_{i,u}),g_ d(X_{i,v})]|\\
&\leq 4\sum_{\begin{matrix}a,b,c,d \in \{1,2\}\\
a+b+c+d=6
\end{matrix}}\sum_{\begin{matrix}
s,t,u,v=1\\
s+t+u+v\leq T
\end{matrix}}|\mathbb{E}[g_a(X_{i,s}),g_b(X_{i,s+t}),g_c(X_{i,s+t+u}),g_ d(X_{i,s+t+u+v})]|\\
&\leq 4T\sum_{\begin{matrix}a,b,c,d \in \{1,2\}\\
a+b+c+d=6
\end{matrix}}\left(\sum_{u,v\leq t}|\mathbb{E}[g_a(X_{i,s}),g_b(X_{i,s+t}),g_c(X_{i,s+t+u}),g_ d(X_{i,s+t+u+v})]|\right.\\
&+\sum_{t,v \leq u}|\mathbb{E}[g_a(X_{i,s}),g_b(X_{i,s+t}),g_c(X_{i,s+t+u}),g_ d(X_{i,s+t+u+v})]|\\
&\left.+\sum_{t,u\leq v}|\mathbb{E}[g_a(X_{i,s}),g_b(X_{i,s+t}),g_c(X_{i,s+t+u}),g_ d(X_{i,s+t+u+v})]|\right)=B_1+B_2+B_3.
\end{align*}
Proposition \ref{eq:eqmomentboundeddifferent} yields
\begin{align*}
B_1\leq C_{85}T\sum_{t=1}t^2 (1+t)^{-b}=TC_{86}\geq B_3
\end{align*}
and
\begin{align*}
B_2&\leq TC_{86}+C_{87} \sum_{\begin{matrix}a,b,c,d \in \{1,2\}\\
a+b+c+d=6
\end{matrix}}\sum_{u,v=1}^T\mathbb{E}[g_a(X_{i,s}),g_b(X_{i,s+t})]\sum_{u,v=1}^T \mathbb{E}[g_c(X_{i,s}),g_d(X_{i,s+t})]\\
&\leq TC_{86}+C_{87}T^2.
\end{align*}
If $r=m+1$ or $n=m+1$ one can use (\ref{trivial2}) which yields $A_1\leq C_{88} T^3$ respectively $A_1\leq C_{88} T^4$ if $r=n=m+1.$
Together we obtain
\begin{align*}
\mathbb{E}\left(\left[\hat{v}_i^2-\check{v}_i^2\right]^2\right)\leq C_{89}T^{-2\beta}+C_{90}T^{-2\beta(m+1)+1}+C_{91}T^{-4\beta(m+1)+2},
\end{align*}
from now on we follow the proof of Theorem \ref{T2}.
\end{proof}
\begin{proof}[Proof of Theorem 6]
First we show that $I_{\{X_i\leq x\}}$ is P-NED for $x\in \mathbb{R}$ with approximating constants $\tilde{a}_k=a_k^\kappa+\Phi(a_k^\kappa)a_k$ and error function $\tilde{\Phi}(x)=I_{(0,1)}(x).$ Denote therefore $R=\sup_{x\in \mathbb{R}}f(x),$ then
\begin{align*}
P(|I_{\{X_0\leq x\}}&-I_{\{f_k(Z_{-k},\ldots,Z_k)\leq x\}}|>\epsilon)\\
&=P(|I_{\{X_0\leq x\}}-I_{\{f_k(Z_{-k},\ldots,Z_k)\leq x\}}|>\epsilon||X_0-x|\leq a_k^\kappa)P(|X_0-x|\leq a_k^\kappa)\\
&+P(|I_{\{X_0\leq x\}}-I_{\{f_k(Z_{-k},\ldots,Z_k)\geq x\}}|>\epsilon||X_0-x|\leq a_k^\kappa)P(|X_0-x|\geq a_k^\kappa)\\
&\leq 2a_k^\kappa R+\Phi(a_k^\kappa)a_k
\end{align*} 
and above probability is 0 for $\epsilon\geq 1$, so we can choose $\tilde{\Phi}(x)=I_{(0,1)}(x)$ for $x>0.$ Now we follow the proof in \cite{serfling}, except for using Markov's inequality in combination with Proposition \ref{eq:Bernsteinbound} instead of Hoeffding's inequality. Let $F$ denote the distribution function of $X_1,$ then
\begin{align}\label{medianzer}
P(|\hat{\mu}-\mu|>\epsilon)=P(\hat{\mu}>\mu+\epsilon)+P(\hat{\mu}<\mu-\epsilon).
\end{align}
For the first summand in (\ref{medianzer}) one has
\begin{align*}
P(\hat{\mu}>\mu+\epsilon)&=P(1/2>\hat{F}_n(\mu+\epsilon))=P\left(\sum_{t=1}^T I_{\{X_t>\mu+\epsilon\}}>T/2\right)\\
&\leq P\left(\sum_{t=1}^T \left\{I_{\{X_t>\mu+\epsilon\}}-\mathbb{E}\left[I_{\{X_t>\mu+\epsilon\}}\right]\right\}>T[F(\mu+\epsilon)-1/2] \right)\\
&\leq \frac{G_2T^{\lfloor p/2 \rfloor}}{T^p F(\mu+\epsilon-1/2)^p}\leq \frac{G_2}{T^{p/2}\epsilon^p M^p}
\end{align*} 
where the last inequality is due to the mean value theorem. The same calculation for the second summand in (\ref{medianzer}) yields (\ref{medianaus}). Now we proof the inequality for the MAD in the same way as the one for the median. Denote $G$ the distribution function of $Y_1,$ then
\begin{align}\label{madzer}
P(|\hat{\sigma}-\sigma|>\epsilon)=P(\hat{\sigma}>\sigma+\epsilon)+P(\hat{\sigma}<\sigma-\epsilon).
\end{align}
For the first summand in (\ref{madzer}) we have:
\begin{align*}
P(\hat{\sigma}>\sigma+\epsilon)&=P\left(\sum_{t=1}^T I_{\{Y_i>\sigma+\epsilon\}>T/2}\right)\\
&\leq P\left(\sum_{t=1}^T I_{\{|X_t-\hat{\mu}|>\sigma+\epsilon\}}>T/2\Big||\hat{\mu}-\mu|<\epsilon/2\right)P(|\hat{\mu}-\mu|<\epsilon/2)\\
&+P\left(\sum_{t=1}^T I_{\{|X_t-\hat{\mu}|>\sigma+\epsilon\}}>T/2\Big||\hat{\mu}-\mu|>\epsilon/2\right)P(|\hat{\mu}-\mu|>\epsilon/2)\\
&\leq P\left(\sum_{t=1}^T I_{\{|X_t-\hat{\mu}|>\sigma+\epsilon/2\}}>T/2\right)+P(|\hat{\mu}-\mu|>\epsilon/2)\\
&\leq \frac{\tilde{G}_2}{T^{p/2}(\epsilon/2)^p \tilde{M}^p}+\frac{G_2}{T^{p/2}(\epsilon/2)^p M^p}
\end{align*}
which one also obtains for the second summand in (\ref{madzer}).
\end{proof}

\end{document}